\documentclass[11pt]{article}
\usepackage{amssymb,amsthm,amsmath}
\usepackage[dvipdfm,colorlinks,breaklinks,
pdfstartview={FitH -32768},
bookmarks,bookmarksnumbered,bookmarksopen,
pdftitle={Homology and topological full groups of 
etale groupoids on totally disconnected spaces},
pdfkeywords={2010 Mathematics Subject Classification: 
37B05, 22A22, 46L80, 19D55},
pdfauthor={Hiroki Matui}]{hyperref}

\newtheorem{thm}{Theorem}[section]
\newtheorem{lem}[thm]{Lemma}
\newtheorem{prop}[thm]{Proposition}
\newtheorem{cor}[thm]{Corollary}

\theoremstyle{definition}
\newtheorem{df}[thm]{Definition}
\newtheorem{rem}[thm]{Remark}

\renewcommand{\phi}{\varphi}

\renewcommand{\ker}{\operatorname{Ker}}

\newcommand{\Aut}{\operatorname{Aut}}
\newcommand{\Inn}{\operatorname{Inn}}
\newcommand{\Out}{\operatorname{Out}}
\newcommand{\Ad}{\operatorname{Ad}}
\newcommand{\Ima}{\operatorname{Im}}
\newcommand{\id}{\operatorname{id}}
\newcommand{\Homeo}{\operatorname{Homeo}}
\newcommand{\Coker}{\operatorname{Coker}}
\newcommand{\supp}{\operatorname{supp}}

\newcommand{\ep}{\varepsilon}

\newcommand{\N}{\mathbb{N}}
\newcommand{\Z}{\mathbb{Z}}
\newcommand{\R}{\mathbb{R}}
\newcommand{\T}{\mathbb{T}}

\title{Homology and topological full groups of \\
\'etale groupoids on totally disconnected spaces
\thanks{2010 Mathematics Subject Classification: 
37B05, 22A22, 46L80, 19D55}}
\author{Hiroki Matui \\
Graduate School of Science \\
Chiba University \\
Inage-ku, Chiba 263-8522, Japan}
\date{}

\begin{document}
\maketitle

\begin{abstract}
For almost finite groupoids, 
we study how their homology groups reflect 
dynamical properties of their topological full groups. 
It is shown that 
two clopen subsets of the unit space has the same class in $H_0$ 
if and only if there exists an element in the topological full group 
which maps one to the other. 
It is also shown that a natural homomorphism, called the index map, 
from the topological full group to $H_1$ is surjective and 
any element of the kernel can be written 
as a product of four elements of finite order. 
In particular, the index map induces a homomorphism 
from $H_1$ to $K_1$ of the groupoid $C^*$-algebra. 
Explicit computations of homology groups of AF groupoids 
and \'etale groupoids arising from subshifts of finite type are also given. 
\end{abstract}

\section{Introduction}

\'Etale groupoids play an important role 
in the theory of both topological dynamics and operator algebras. 
Among other things, 
their (co)homology theory and $K$-theory of the associated $C^*$-algebras 
have attracted significant interest. 
This paper analyses 
how the homology groups $H_*(G)$ reflect 
dynamical properties of the topological full group $[[G]]$ 
when the groupoid $G$ has 
a compact and totally disconnected unit space $G^{(0)}$. 
The topological full group $[[G]]$ consists of 
all homeomorphisms on $G^{(0)}$ 
whose graph is `contained' in the groupoid $G$ as an open subset 
(Definition \ref{tfg}). 
It corresponds to a natural quotient 
of the group of unitary normalizers of $C(G^{(0)})$ in $C^*_r(G)$. 
(Proposition \ref{exact}). 
With this correspondence, 
we also discuss connections between homology theory and $K$-theory of 
totally disconnected \'etale groupoids. 

The AF groupoids (\cite{R1,Kr,GPS3}) form 
one of the most important classes of 
\'etale groupoids on totally disconnected spaces and 
have already been classified completely up to isomorphism. 
The terminology AF comes from $C^*$-algebra theory and 
means approximately finite. 
In the present paper, 
we introduce a class of `AF-like' groupoids, 
namely almost finite groupoids (Definition \ref{almstfnt}). 
Roughly speaking, 
a totally disconnected \'etale groupoid $G$ is said to be almost finite 
if any compact subset of $G$ is almost contained in an elementary subgroupoid. 
Clearly AF groupoids are almost finite. 
Any transformation groupoid arising from a free action of $\Z^N$ is 
shown to be almost finite (Lemma \ref{ZNisAF}), 
but it is not known whether the same holds for other discrete amenable groups. 
For any almost finite groupoid $G$ we first show that 
two $G$-full clopen subsets of $G^{(0)}$ has the same class in $H_0(G)$ 
if and only if one is mapped to the other by an element in $[[G]]$ 
(Theorem \ref{Hopf2}). 
The latter condition is equivalent to saying that 
the characteristic functions on the clopen sets are unitarily equivalent 
in $C^*_r(G)$ via a unitary normalizer of $C(G^{(0)})$. 
Next, we introduce a group homomorphism from $[[G]]$ to $H_1(G)$ 
and call it the index map. 
When $G$ is almost finite, 
the index map is shown to be surjective (Theorem \ref{surjective}). 
Furthermore, we prove that 
any element of the kernel of the index map is 
a product of four elements of finite order (Theorem \ref{injective}). 
In particular, if $G$ is principal, then 
$H_1(G)$ is isomorphic to $[[G]]/[[G]]_0$, 
where $[[G]]_0$ is the subgroup generated by elements of finite order. 

This paper is organized as follows. 
In Section 2 
we collect notation, definitions and basic facts on \'etale groupoids. 
In Section 3 
we recall the homology theory of \'etale groupoids, 
which was introduced by M. Crainic and I. Moerdijk \cite{CM}. 
We observe that 
homologically similar \'etale groupoids have 
isomorphic homology groups with constant coefficients 
(Proposition \ref{homosim}). 
A variant of the Lindon-Hochschild-Serre spectral sequence is also given. 
In Section 4 
we introduce the notion of Kakutani equivalence 
for \'etale groupoids with compact and totally disconnected unit spaces 
and prove its elementary properties. 
Kakutani equivalent groupoids are shown to be homologically similar 
(Theorem \ref{Kakutani>sim}). 
With the aid of the results of Section 3 and 4, 
we compute the homology groups of the AF groupoids 
(Theorem \ref{AFHomology1}, \ref{AFHomology2}) and 
the \'etale groupoids arising from subshifts of finite type 
(Theorem \ref{CKHomology}). 
Note that the homology groups agree with 
the $K$-groups of the associated groupoid $C^*$-algebras 
for these groupoids. 
In Section 5 
we give a $C^*$-algebraic characterization of Kakutani equivalence 
(Theorem \ref{C*Kakutani}) 
by using a result of J. Renault \cite{R3}. 
Next, we study relationship 
between the topological full group $[[G]]$ and 
the unitary normalizers of $C(G^{(0)})$ in $C^*_r(G)$ 
and show a short exact sequence for them (Proposition \ref{exact}). 
We also study the group of automorphisms of $C^*_r(G)$ 
preserving $C(G^{(0)})$ globally (Proposition \ref{exact2}). 
In Section 6 
the definition of almost finite groupoids is given 
(Definition \ref{almstfnt}), 
and some basic properties are proved. 
Transformation groupoids arising from free $\Z^N$-actions are 
shown to be almost finite (Lemma \ref{ZNisAF}). 
We also prove that $H_0$ of any minimal and almost finite groupoid is 
a simple, weakly unperforated, ordered abelian group 
with the Riesz interpolation property (Proposition \ref{Riesz}). 
The main result of this section is Theorem \ref{Hopf2}, 
which says that 
two clopen subsets of the unit space with the same image in $H_0$ 
are mapped to each other by an element of the topological full group 
when the groupoid is almost finite. 
In Section 7 
we investigate the index map $I:[[G]]\to H_1(G)$. 
For an almost finite groupoid $G$, 
Theorem \ref{surjective} states that $I$ is surjective and 
Theorem \ref{injective} determines the kernel of $I$. 
As a result, 
the existence of a natural homomorphism $\Phi_1$ 
from $H_1(G)$ to $K_1(C^*_r(G))$ is shown (Corollary \ref{Phi1}).

\section{Preliminaries}

The cardinality of a set $A$ is written by $\lvert A\rvert$ and 
the characteristic function on $A$ is written by $1_A$. 
We say that a subset of a topological space is clopen 
if it is both closed and open. 
A topological space is said to be totally disconnected 
if its topology is generated by clopen subsets. 
By a Cantor set, 
we mean a compact, metrizable, totally disconnected space 
with no isolated points. 
It is known that any two such spaces are homeomorphic. 

We say that a continuous map $f:X\to Y$ is \'etale, 
if it is a local homeomorphism, 
i.e. each $x\in X$ has an open neighborhood $U$ such that 
$f(U)$ is open in $Y$ and $f|U$ is a homeomorphism from $U$ to $f(U)$. 
In this article, by an \'etale groupoid 
we mean a locally compact Hausdorff groupoid 
such that the range map is \'etale. 
We refer the reader to \cite{R1,R3} 
for the basic theory of \'etale groupoids. 
For an \'etale groupoid $G$, 
we let $G^{(0)}$ denote the unit space and 
let $s$ and $r$ denote the source and range maps. 
A subset $F\subset G^{(0)}$ is said to be $G$-full, 
if $r^{-1}(x)\cap s^{-1}(F)$ is not empty for any $x\in G^{(0)}$. 
For $x\in G^{(0)}$, 
$G(x)=r(Gx)$ is called the $G$-orbit of $x$. 
When every $G$-orbit is dense in $G^{(0)}$, 
$G$ is said to be minimal. 
For an open subset $F\subset G^{(0)}$, 
the reduction of $G$ to $F$ is $r^{-1}(F)\cap s^{-1}(F)$ and 
denoted by $G|F$. 
The reduction $G|F$ is an \'etale subgroupoid of $G$ in an obvious way. 
A subset $U\subset G$ is called a $G$-set, if $r|U,s|U$ are injective. 
For an open $G$-set $U$, 
we let $\tau_U$ denote the homeomorphism $r\circ(s|U)^{-1}$ 
from $s(U)$ to $r(U)$. 
The isotropy bundle is $G'=\{g\in G\mid r(g)=s(g)\}$. 
We say that $G$ is principal, if $G'=G^{(0)}$. 
A principal \'etale groupoid $G$ can be identified with 
$\{(r(g),s(g))\in G^{(0)}\times G^{(0)}\mid g\in G\}$, 
which is an equivalence relation on $G^{(0)}$. 
Such an equivalence relation is called an \'etale equivalence relation 
(see \cite[Definition 2.1]{GPS2}). 
When the interior of $G'$ is $G^{(0)}$, 
we say that $G$ is essentially principal 
(this is slightly different from Definition II.4.3 of \cite{R1}, 
but the same as the definition given in \cite{R3}). 
For a second countable \'etale groupoid $G$, 
by \cite[Proposition 3.1]{R3}, 
$G$ is essentially principal if and only if 
the set of points of $G^{(0)}$ with trivial isotropy is dense in $G^{(0)}$. 
For an \'etale groupoid $G$, 
we denote the reduced groupoid $C^*$-algebra of $G$ by $C^*_r(G)$ and 
identify $C_0(G^{(0)})$ with a subalgebra of $C^*_r(G)$. 

There are two important examples of \'etale groupoids. 
One is the class of transformation groupoids 
arising from actions of discrete groups. 

\begin{df}
Let $\phi:\Gamma\curvearrowright X$ be 
an action of a countable discrete group $\Gamma$ 
on a locally compact Hausdorff space $X$ by homeomorphisms. 
We let $G_\phi=\Gamma\times X$ and define the following groupoid structure: 
$(\gamma,x)$ and $(\gamma',x')$ are composable 
if and only if $x=\phi^{\gamma'}(x')$, 
$(\gamma,\phi^{\gamma'}(x'))\cdot(\gamma',x')=(\gamma\gamma',x')$ and 
$(\gamma,x)^{-1}=(\gamma^{-1},\phi^\gamma(x))$. 
Then $G_\phi$ is an \'etale groupoid and 
called the transformation groupoid 
arising from $\phi:\Gamma\curvearrowright X$. 
\end{df}

If the action $\phi$ is free 
(i.e. $\{\gamma\in\Gamma\mid\phi^\gamma(x){=}x\}=\{e\}$ for all $x\in X$, 
where $e$ denotes the neutral element), then 
$G_\phi$ is principal. 
The reduced groupoid $C^*$-algebra $C^*_r(G_\phi)$ is naturally 
isomorphic to the reduced crossed product $C^*$-algebra 
$C_0(X)\rtimes_{r,\phi}\Gamma$. 

The other important class is AF groupoids 
(\cite[Definition III.1.1]{R1}, \cite[Definition 3.7]{GPS3}). 

\begin{df}
Let $G$ be a second countable \'etale groupoid 
whose unit space is compact and totally disconnected. 
\begin{enumerate}
\item We say that $K\subset G$ is an elementary subgroupoid 
if $K$ is a compact open principal subgroupoid of $G$ 
such that $K^{(0)}=G^{(0)}$. 
\item We say that $G$ is an AF groupoid 
if it can be written as an increasing union of elementary subgroupoids. 
\end{enumerate}
\end{df}

An AF groupoid is principal by definition, and so 
it can be identified with an equivalence relation on the unit space. 
When $G$ is an AF groupoid, 
the reduced groupoid $C^*$-algebra $C^*_r(G)$ is an AF algebra. 
W. Krieger \cite{Kr} showed that 
two AF groupoids $G_1$ and $G_2$ are isomorphic if and only if 
$K_0(C^*_r(G_1))$ and $K_0(C^*_r(G_2))$ are isomorphic 
as ordered abelian groups with distinguished order units 
(see \cite[Definition 1.1.8]{Ro} 
for the definition of ordered abelian groups and order units). 
For classification of minimal AF groupoids up to orbit equivalence, 
we refer the reader to \cite{GPS1,GPS3}. 

We introduce the notion of full groups and topological full groups 
for \'etale groupoids. 

\begin{df}\label{tfg}
Let $G$ be an \'etale groupoid whose unit space $G^{(0)}$ is compact. 
\begin{enumerate}
\item The set of all $\gamma\in\Homeo(G^{(0)})$ such that 
for every $x\in G^{(0)}$ there exists $g\in G$ 
satisfying $r(g)=x$ and $s(g)=\gamma(x)$ 
is called the full group of $G$ and denoted by $[G]$. 
\item The set of all $\gamma\in\Homeo(G^{(0)})$ for which 
there exists a compact open $G$-set $U$ satisfying $\gamma=\tau_U$ 
is called the topological full group of $G$ and denoted by $[[G]]$. 
\end{enumerate}
Obviously $[G]$ is a subgroup of $\Homeo(G^{(0)})$ and 
$[[G]]$ is a subgroup of $[G]$. 
\end{df}

For a minimal homeomorphism $\phi$ on a Cantor set $X$, 
its full group $[\phi]$ and topological full group $\tau[\phi]$ were 
defined in \cite{GPS2}. 
One can check that 
$[\phi]$ and $\tau[\phi]$ are equal to $[G_\phi]$ and $[[G_\phi]]$ 
respectively, 
where $G_\phi=\Z\times X$ is the transformation groupoid arising from $\phi$. 
Moreover, for an \'etale equivalence relation 
on a compact metrizable and totally disconnected space, 
its topological full group was introduced in \cite{Msrtfg} and 
the above definition is an adaptation of it 
for a groupoid not necessarily principal. 
When $G$ is the \'etale groupoid 
arising from a subshift of finite type, 
$[[G]]$ and its connection with $C^*$-algebras were studied 
by K. Matsumoto (\cite{Matsu}).

\section{Homology theory for \'etale groupoids}

We briefly recall homology theory for \'etale groupoids 
which was studied in \cite{CM}. 
In \cite{CM} homology groups are defined for sheaves on the unit space and 
discussed from various viewpoints by using methods of algebraic topology. 
Here, we restrict our attention to the case of constant coefficients and 
introduce homology groups in an elementary way, 
especially for people who are not familiar with algebraic topology.

\subsection{Homology groups of \'etale groupoids}

Let $A$ be a topological abelian group. 
For a locally compact Hausdorff space $X$, we denote by $C_c(X,A)$ 
the set of $A$-valued continuous functions with compact support. 
When $X$ is compact, we simply write $C(X,A)$. 
With pointwise addition, $C_c(X,A)$ is an abelian group. 
Let $\pi:X\to Y$ be an \'etale map between locally compact Hausdorff spaces. 
For $f\in C_c(X,A)$, we define a map $\pi_*(f):Y\to A$ by 
\[
\pi_*(f)(y)=\sum_{\pi(x)=y}f(x). 
\]
It is not so hard to see that $\pi_*(f)$ belongs to $C_c(Y,A)$ and 
that $\pi_*$ is a homomorphism from $C_c(X,A)$ to $C_c(Y,A)$. 
Besides, if $\pi':Y\to Z$ is another \'etale map to 
a locally compact Hausdorff space $Z$, then 
one can check $(\pi'\circ\pi)_*=\pi'_*\circ\pi_*$ in a direct way. 

Let $G$ be an \'etale groupoid and 
let $G^{(0)}$ be the unit space. 
We let $s$ and $r$ denote the source and range maps. 
For $n\in\N$, we write $G^{(n)}$ 
for the space of composable strings of $n$ elements in $G$, that is, 
\[
G^{(n)}=\{(g_1,g_2,\dots,g_n)\in G^n\mid
s(g_i)=r(g_{i+1})\text{ for all }i=1,2,\dots,n{-}1\}. 
\]
For $i=0,1,\dots,n$, 
we let $d_i:G^{(n)}\to G^{(n-1)}$ be a map defined by 
\[
d_i(g_1,g_2,\dots,g_n)=\begin{cases}
(g_2,g_3,\dots,g_n) & i=0 \\
(g_1,\dots,g_ig_{i+1},\dots,g_n) & 1\leq i\leq n{-}1 \\
(g_1,g_2,\dots,g_{n-1}) & i=n. 
\end{cases}
\]
Clearly $d_i$ is \'etale. 
Let $A$ be a topological abelian group. 
Define the homomorphisms 
$\delta_n:C_c(G^{(n)},A)\to C_c(G^{(n-1)},A)$ 
by 
\[
\delta_1=s_*-r_* \qquad\text{ and }\qquad
\delta_n=\sum_{i=0}^n(-1)^id_{i*}. 
\]
It is easy to see that 
\[
0\stackrel{\delta_0}{\longleftarrow}
C_c(G^{(0)},A)\stackrel{\delta_1}{\longleftarrow}
C_c(G^{(1)},A)\stackrel{\delta_2}{\longleftarrow}
C_c(G^{(2)},A)\stackrel{\delta_3}{\longleftarrow}\dots
\]
is a chain complex. 

\begin{df}\label{homology}
We let $H_n(G,A)$ be the homology groups of the chain complex above, 
i.e. $H_n(G,A)=\ker\delta_n/\Ima\delta_{n+1}$, 
and call them the homology groups of $G$ with constant coefficients $A$. 
When $A=\Z$, we simply write $H_n(G)=H_n(G,\Z)$. 
In addition, we define 
\[
H_0(G)^+=\{[f]\in H_0(G)\mid f(u)\geq0\text{ for all }u\in G^{(0)}\}, 
\]
where $[f]$ denotes the equivalence class of $f\in C_c(G^{(0)},A)$. 
\end{df}

\begin{rem}
The pair $(H_0(G),H_0(G)^+)$ is not necessarily an ordered abelian group 
in general, 
because $H_0(G)^+\cap({-}H_0(G)^+)$ may not equal $\{0\}$. 
In fact, 
when $G$ is the \'etale groupoid arising from the full shift over $N$ symbols, 
$H_0(G)^+=H_0(G)\cong\Z/(N{-}1)\Z$. 
See Theorem \ref{CKHomology}. 
\end{rem}

Let $\phi:\Gamma\curvearrowright X$ be an action of 
a discrete group $\Gamma$ on a locally compact Hausdorff space $X$ 
by homeomorphisms. 
With pointwise addition $C_c(X,A)$ is an abelian group, 
and $\Gamma$ acts on it by translation. 
One can check that 
$H_n(G_\phi,A)$ is canonically isomorphic to $H_n(\Gamma,C_c(X,A))$, 
the homology of $\Gamma$ with coefficients in $C_c(X,A)$. 
Under the identification of $G^{(0)}$ with $X$, 
the image of $\delta_1$ is equal to the subgroup of $C_c(X,A)$ 
generated by 
\[
\{f-f\circ\phi^\gamma\mid f\in C_c(X,A), \ \gamma\in\Gamma\}, 
\]
and $H_0(G_\phi,A)$ is equal to the quotient of $C_c(X,A)$ by this subgroup. 

Suppose that $G^{(0)}$ is compact, metrizable and totally disconnected. 
The canonical inclusion $\iota:C(G^{(0)})\to C^*_r(G)$ induces 
a homomorphism $K_0(\iota):K_0(C(G^{(0)}))\to K_0(C^*_r(G))$. 
The $K_0$-group of $C(G^{(0)})$ is naturally identified with $C(G^{(0)},\Z)$. 
If $U$ is a compact open $G$-set, 
then $u=1_U$ is a partial isometry in $C^*_r(G)$ 
satisfying $u^*u=1_{s(U)}$ and $uu^*=1_{r(U)}$. 
Hence $(K_0(\iota)\circ\delta_1)(1_U)$ is zero. 
This means that 
the image of $\delta_1$ is contained in the kernel of $K_0(\iota)$, 
because $G$ has a countable base of compact open $G$-sets. 
It follows that 
we obtain a homomorphism $\Phi_0:H_0(G)\to K_0(C^*_r(G))$ 
such that $\Phi_0([f])=K_0(\iota)(f)$. 
It is natural to ask if the homomorphism $\Phi_0$ is injective or not, 
but we do not know the answer 
even when $G$ is the transformation groupoid arising from 
a free minimal action of $\Z^N$. 
In Section 7, we will show that 
there also exists a natural homomorphism $\Phi_1:H_1(G)\to K_1(C^*_r(G))$ 
under the assumption that $G$ is almost finite (Corollary \ref{Phi1}). 

Let $\phi:\Z^N\curvearrowright X$ be an action of $\Z^N$ on a Cantor set $X$. 
As mentioned above, 
$H_*(G_\phi)$ is isomorphic to the group homology $H_*(\Z^N,C(X,\Z))$, 
and hence to the group cohomology $H^*(\Z^N,C(X,\Z))$ by Poincar\'e duality. 
When $N=1$, 
it is straightforward to check that 
$H_i(G_\phi)$ is isomorphic to $K_i(C^*_r(G_\phi))$ for $i=0,1$. 
It is natural to ask whether the isomorphisms 
\[
K_0(C^*_r(G_\phi))\cong\bigoplus_iH_{2i}(G_\phi),\quad 
K_1(C^*_r(G_\phi))\cong\bigoplus_iH_{2i+1}(G_\phi)
\tag{$*$}
\]
hold for general $N$. 
As shown in \cite{FH}, there exists a spectral sequence 
\[
E_2^{p,q}\Rightarrow K_{p+q+N}(C^*_r(G_\phi))\quad\text{with}\quad
E_2^{p,q}=\begin{cases}H^p(\Z^N,C(X,\Z))&q\text{ is even} \\
0&q\text{ is odd,}\end{cases}
\]
and if the (co)homology groups were always torsion-free, then 
the isomorphisms $(*)$ would follow from this spectral sequence. 
However, it turns out that 
there exists a free minimal $\Z^N$-action $\phi$ 
which contains torsion in its (co)homology (\cite{GHK,Mtc}). 
Nevertheless, for certain classes of $\Z^N$-actions, 
it is known that the isomorphisms $(*)$ hold. 
We refer the reader to \cite{AP}, \cite{GHK} and \cite{SB} 
for detailed information. 
We also remark that 
the isomorphisms $(*)$ hold for AF groupoids and \'etale groupoids 
arising from subshifts of finite type 
(see Theorem \ref{AFHomology1}, \ref{AFHomology2}, \ref{CKHomology}).

\subsection{Homological similarity}

In this subsection we introduce the notion of homological similarity 
(Definition \ref{similarity}) and 
prove that homologically similar groupoids have isomorphic homology 
(Proposition \ref{homosim}). 
A variant of Lindon-Hochschild-Serre spectral sequence is also given 
(Theorem \ref{LHS}). 

To begin with, we would like to consider functoriality of $H_n(G,A)$. 

\begin{df}
A map $\rho:G\to H$ between \'etale groupoids is called a homomorphism, 
if $\rho$ is a continuous map satisfying 
\[
(g,g')\in G^{(2)}\Rightarrow
(\rho(g),\rho(g'))\in H^{(2)}\text{ and }
\rho(g)\rho(g')=\rho(gg'). 
\]
We emphasize that continuity is already built in the definition. 
\end{df}

Compactly supported cohomology of spaces is 
covariant along local homeomorphisms and contravariant along proper maps. 
Analogous properties hold for homology of \'etale groupoids. 
Let $\rho:G\to H$ be a homomorphism between \'etale groupoids. 
We let $\rho^{(0)}$ denote the restriction of $\rho$ to $G^{(0)}$ and 
$\rho^{(n)}$ denote the restriction of the $n$-fold product 
$\rho\times\rho\times\dots\times\rho$ to $G^{(n)}$. 
One can easily see that the following three conditions are equivalent. 
\begin{enumerate}
\item $\rho^{(0)}$ is \'etale (i.e. a local homeomorphism). 
\item $\rho$ is \'etale. 
\item $\rho^{(n)}$ is \'etale for all $n\in\N$. 
\end{enumerate}
When $\rho$ is \'etale, 
$\rho^{(n)}_*:C_c(G^{(n)},A)\to C_c(H^{(n)},A)$ are homomorphisms 
commuting with the boundary operators $\delta_n$. 
It follows that we obtain a homomorphism 
\[
H_n(\rho):H_n(G,A)\to H_n(H,A). 
\]
If $\rho$ is proper, then 
one obtains a homomorphism from $C_c(H^{(n)},A)$ to $C_c(G^{(n)},A)$ 
by pullback, and hence a homomorphism 
\[
H^*_n(\rho):H_n(H,A)\to H_n(G,A). 
\]

The following is a variant of (continuous) similarity introduced in \cite{R1}. 
See also \cite[Proposition 3.8]{CM} and \cite[2.1.3]{FHK}. 

\begin{df}\label{similarity}
Let $G,H$ be \'etale groupoids. 
\begin{enumerate}
\item Two homomorphisms $\rho,\sigma$ from $G$ to $H$ are 
said to be similar 
if there exists an continuous map $\theta:G^{(0)}\to H$ such that 
\[
\theta(r(g))\rho(g)=\sigma(g)\theta(s(g))
\]
for all $g\in G$. 
Note that if $\rho$ and $\sigma$ are \'etale, then 
$\theta$ becomes automatically \'etale. 
\item The two groupoids $G$ and $H$ are said to be homologically similar 
if there exist \'etale homomorphisms $\rho:G\to H$ and $\sigma:H\to G$ 
such that $\sigma\circ\rho$ is similar to $\id_G$ and 
$\rho\circ\sigma$ is similar to $\id_H$. 
\end{enumerate}
\end{df}

\begin{prop}\label{homosim}
Let $G,H$ be \'etale groupoids. 
\begin{enumerate}
\item If \'etale homomorphisms $\rho,\sigma$ from $G$ to $H$ are similar, then 
$H_n(\rho)=H_n(\sigma)$. 
\item If $G$ and $H$ are homologically similar, 
then they have isomorphic homology with constant coefficients $A$. 
Moreover when $A=\Z$, the isomorphism maps $H_0(G)^+$ onto $H_0(H)^+$. 
\end{enumerate}
\end{prop}
\begin{proof}
It suffices to show (1). 
There exists an \'etale map $\theta:G^{(0)}\to H$ such that 
$\theta(r(g))\rho(g)=\sigma(g)\theta(s(g))$ for all $g\in G$. 
For each $n\in\N\cup\{0\}$ 
we construct a homomorphism $h_n:C_c(G^{(n)},A)\to C_c(H^{(n+1)},A)$ 
as follows. 
First we put $h_0=\theta_*$. 
For $n\in\N$, we let $h_n=\sum_{j=0}^n(-1)^jk_{j*}$, 
where $k_j:G^{(n)}\to H^{(n+1)}$ is defined by 
\[
k_j(g_1,g_2,\dots,g_n)=\begin{cases}
(\theta(r(g_1)),\rho(g_1),\rho(g_2),\dots,\rho(g_n)) & j=0 \\
(\sigma(g_1),\dots,\sigma(g_j),\theta(s(g_j)),
\rho(g_{j+1}),\dots,\rho(g_n)) & 1\leq j\leq n-1 \\
(\sigma(g_1),\sigma(g_2),\dots,\sigma(g_n),\theta(s(g_n))) & j=n. 
\end{cases}
\]
It is straightforward to verify 
$\delta_1\circ h_0=\rho^{(0)}_*-\sigma^{(0)}_*$ and 
\[
\delta_{n+1}\circ h_n+h_{n-1}\circ\delta_n
=\rho^{(n)}_*-\sigma^{(n)}_*
\]
for all $n\in\N$. 
Hence we get $H_n(\rho)=H_n(\sigma)$. 
\end{proof}

\begin{thm}\label{reduction}
Let $G$ be an \'etale groupoid and 
let $F\subset G^{(0)}$ be an open $G$-full subset. 
\begin{enumerate}
\item If there exists a continuous map $\theta:G^{(0)}\to G$ such that 
$r(\theta(x))=x$ and $s(\theta(x))\in F$ for all $x\in G^{(0)}$, then 
$G$ is homologically similar to $G|F$. 
\item Suppose that $G^{(0)}$ is $\sigma$-compact and totally disconnected. 
Then $G$ is homologically similar to $G|F$. 
\end{enumerate}
\end{thm}
\begin{proof}
(1) 
Notice that $\theta$ is \'etale. 
Let $\rho:G\to G|F$ and $\sigma:G|F\to G$ be \'etale homomorphisms 
defined by $\rho(g)=\theta(r(g))^{-1}g\theta(s(g))$ and $\sigma(g)=g$. 
It is easy to see that 
$\rho\circ\sigma$ is similar to the identity on $G|F$ and 
that $\sigma\circ\rho$ is similar to the identity on $G$ 
via the map $\theta$. 
Hence $G$ and $G|F$ are homologically similar. 

(2) There exists a countable family of compact open $G$-sets $\{U_n\}_n$ 
such that $\{r(U_n)\}_n$ covers $G$ and $s(U_n)\subset F$. 
Define compact open $G$-sets $V_1,V_2,\dots$ inductively by $V_1=U_1$ and 
\[
V_n=U_n\setminus r^{-1}(r(V_1\cup\dots\cup V_{n-1})). 
\]
We can define $\theta:G^{(0)}\to G$ 
by $\theta(x)=(r|V_n)^{-1}(x)$ for $x\in V_n$. 
Clearly $\theta$ satisfies the assumption of (1), 
and so the proof is completed. 
\end{proof}

We recall from \cite{R1} 
the notion of skew products and semi-direct products of \'etale groupoids. 
Let $G$ be an \'etale groupoid and 
let $\Gamma$ be a countable discrete group. 
When $\rho:G\to\Gamma$ is a homomorphism, 
the skew product $G\times_\rho\Gamma$ is 
$G\times\Gamma$ with the following groupoid structure: 
$(g,\gamma)$ and $(g',\gamma')$ is composable 
if and only if $g$ and $g'$ are composable and $\gamma\rho(g)=\gamma'$, 
$(g,\gamma)\cdot(g',\gamma\rho(g))=(gg',\gamma)$ and 
$(g,\gamma)^{-1}=(g^{-1},\gamma\rho(g))$. 
We can define an action $\hat\rho:\Gamma\curvearrowright G\times_\rho\Gamma$ 
by $\hat\rho^\gamma(g',\gamma')=(g',\gamma\gamma')$. 

When $\phi:\Gamma\curvearrowright G$ is an action of $\Gamma$ on $G$, 
the semi-direct product $G\rtimes_\phi\Gamma$ is 
$G\times\Gamma$ with the following groupoid structure: 
$(g,\gamma)$ and $(g',\gamma')$ is composable 
if and only if $g$ and $\phi^\gamma(g')$ are composable, 
$(g,\gamma)\cdot(g',\gamma')=(g\phi^\gamma(g'),\gamma\gamma')$ and 
$(g,\gamma)^{-1}=(\phi^{\gamma^{-1}}(g^{-1}),\gamma^{-1})$. 
There exists a natural homomorphism $\tilde\phi:G\rtimes_\phi\Gamma\to\Gamma$ 
defined by $\tilde\phi(g,\gamma)=\gamma$. 
The following proposition can be shown 
in a similar fashion to \cite[I.1.8]{R1} 
by using Theorem \ref{reduction} (1). 

\begin{prop}\label{Takai}
Let $G$ be an \'etale groupoid and 
let $\Gamma$ be a countable discrete group. 
\begin{enumerate}
\item When $\rho:G\to\Gamma$ is a homomorphism, 
$(G\times_\rho\Gamma)\rtimes_{\hat\rho}\Gamma$ is 
homologically similar to $G$. 
\item When $\phi:\Gamma\curvearrowright G$ is an action, 
$(G\rtimes_\phi\Gamma)\times_{\tilde\phi}\Gamma$ is 
homologically similar to $G$. 
\end{enumerate}
\end{prop}

For skew products and semi-direct products, 
the following Lindon-Hochschild-Serre spectral sequences exist. 
This will be used later 
for a computation of the homology groups of 
\'etale groupoids arising from subshifts of finite type. 

\begin{thm}\label{LHS}
Let $G$ be an \'etale groupoid and 
let $\Gamma$ be a countable discrete group. 
Let $A$ be a topological abelian group. 
\begin{enumerate}
\item Suppose that $\rho:G\to\Gamma$ is a homomorphism. 
Then there exists a spectral sequence: 
\[
E^2_{p,q}=H_p(\Gamma,H_q(G\times_\rho\Gamma,A))\Rightarrow
H_{p+q}(G,A), 
\]
where $H_q(G\times_\rho\Gamma,A)$ is regarded as a $\Gamma$-module 
via the action $\hat\rho:\Gamma\curvearrowright G\times_\rho\Gamma$. 
\item Suppose that $\phi:\Gamma\curvearrowright G$ is an action. 
Then there exists a spectral sequence: 
\[
E^2_{p,q}=H_p(\Gamma,H_q(G,A))\Rightarrow
H_{p+q}(G\rtimes_\phi\Gamma,A), 
\]
where $H_q(G,A)$ is regarded as a $\Gamma$-module via the action $\phi$. 
\end{enumerate}
\end{thm}
\begin{proof}
(1) is a special case of \cite[Theorem 4.4]{CM}. 
(2) immediately follows from (1) and Proposition \ref{Takai} (2). 
\end{proof}

We remark that 
similar spectral sequences exist for cohomology of \'etale groupoids, too.

\section{Kakutani equivalence}

In this section we introduce the notion of Kakutani equivalence 
for \'etale groupoids whose unit spaces are compact and totally disconnected. 
We also compute the homology groups of AF groupoids and 
\'etale groupoids arising from subshifts of finite type.

\subsection{Kakutani equivalence}

\begin{df}
Let $G_i$ be an \'etale groupoid 
whose unit space is compact and totally disconnected for $i=1,2$. 
When there exists a $G_i$-full clopen subset $Y_i\subset G_i^{(0)}$ 
for $i=1,2$ and $G_1|Y_1$ is isomorphic to $G_2|Y_2$, 
we say that $G_1$ is Kakutani equivalent to $G_2$. 
\end{df}

It will be proved in Lemma \ref{Kakutaniequiv} that 
the Kakutani equivalence is really an equivalence relation. 

\begin{rem}
In the case of transformation groupoids arising from $\Z$-actions, 
the Kakutani equivalence defined above is weaker 
than the Kakutani equivalence for $\Z$-actions introduced 
in \cite{GPS1}. 
Indeed, 
for minimal homeomorphisms $\phi_1,\phi_2$ on Cantor sets, 
the \'etale groupoids associated with them are 
Kakutani equivalent in the sense above if and only if 
$\phi_1$ is Kakutani equivalent to either of $\phi_2$ and $\phi_2^{-1}$ 
in the sense of \cite[Definition 1.7]{GPS1}. 
See also \cite[Theorem 2.4]{GPS1} and \cite{BT}. 
\end{rem}

Let $G$ be an \'etale groupoid 
whose unit space is compact and totally disconnected. 
For $f\in C(G^{(0)},\Z)$ with $f\geq0$, we let 
\[
G_f=\{(g,i,j)\in G\times\Z\times\Z\mid
0\leq i\leq f(r(g)),\ 0\leq j\leq f(s(g))\}
\]
and equip $G_f$ with the induced topology 
from the product topology on $X\times\Z\times\Z$. 
The groupoid structure of $G_f$ is given as follows: 
\[
G_f^{(0)}=\{(x,i,i)\mid x\in G^{(0)}, \ 0\leq i\leq f(x)\}, 
\]
$(g,i,j)^{-1}=(g^{-1},j,i)$, 
two elements $(g,i,j)$ and $(h,k,l)$ are composable 
if and only if $s(g)=r(h),j=k$ and 
the product is $(g,i,j)(h,j,l)=(gh,i,l)$. 
It is easy to see that 
$G_f$ is an \'etale groupoid and 
the clopen subset $\{(x,0,0)\in G_f^{(0)}\mid x\in G^{(0)}\}$ 
of $G_f^{(0)}$ is $G_f$-full. 

\begin{lem}\label{tower1}
Let $G$ be an \'etale groupoid 
whose unit space is compact and totally disconnected and 
let $Y\subset G^{(0)}$ be a $G$-full clopen subset. 
There exists $f\in C(Y,\Z)$ and an isomorphism $\pi:(G|Y)_f\to G$ 
such that $\pi(g,0,0)=g$ for all $g\in G|Y$. 
\end{lem}
\begin{proof}
We put $X=G^{(0)}$ for notational convenience. 
For any $x\in X\setminus Y$, 
there exists $g\in r^{-1}(x)\cap s^{-1}(Y)$, because $Y$ is $R$-full. 
We can choose a compact open $G$-set $U_x$ containing $g$ 
so that $r(U_x)\subset X\setminus Y$ and $s(U_x)\subset Y$. 
The family of clopen subsets $\{r(U_x)\mid x\in X\setminus Y\}$ 
forms an open covering of $X\setminus Y$, 
and so we can find $x_1,x_2,\dots,x_n\in X\setminus Y$ such that 
$r(U_{x_1})\cup r(U_{x_2})\cup\dots\cup r(U_{x_n})=X\setminus Y$. 
Define compact open $G$-sets $V_1,V_2,\dots,V_n$ inductively by 
\[
V_1=U_{x_1} \quad \text{and} \quad
V_k=U_{x_k}\setminus r^{-1}(r(V_1\cup\dots\cup V_{k-1})). 
\]
Then $r(V_1),r(V_2),\dots,r(V_n)$ are mutually disjoint and 
their union is equal to $X\setminus Y$. 
For each subset $\lambda\subset\{1,2,\dots,n\}$, 
we fix a bijection 
$\alpha_\lambda:\{k\in\N\mid k\leq\lvert\lambda\rvert\}\to\lambda$. 
For $y\in Y$, 
put $\lambda(y)=\{k\in\{1,2,\dots,n\}\mid y\in s(V_k)\}$. 
We define $f\in C(Y,\Z)$ by $f(y)=\lvert\lambda(y)\rvert$. 
Since each $s(V_k)$ is clopen, $f$ is continuous. 
We further define $\theta:(G|Y)_f^{(0)}\to G$ by 
\[
\theta(y,i,i)=\begin{cases}
y & \text{if }i=0 \\
(s|V_l)^{-1}(y) & \text{otherwise}, \end{cases}
\]
where $l=\alpha_{\lambda(y)}(i)$. 
It is not so hard to see that 
\[
\pi(g,i,j)=\theta(r(g),i,i)\cdot g\cdot\theta(s(g),j,j)^{-1}
\]
gives an isomorphism from $(G|Y)_f$ to $G$. 
\end{proof}

\begin{lem}\label{tower2}
Let $G$ be an \'etale groupoid 
whose unit space is compact and totally disconnected and 
let $Y,Y'\subset G^{(0)}$ be $G$-full clopen subsets. 
Then $G|Y$ and $G|Y'$ are Kakutani equivalent. 
\end{lem}
\begin{proof}
By Lemma \ref{tower1}, 
there exists $f\in C(Y,\Z)$ and an isomorphism $\pi:(G|Y)_f\to G$ such that 
$\pi(g,0,0)=g$ for all $g\in G|Y$. 
Define a clopen subset $Z\subset Y$ by 
\[
Z=\{y\in Y\mid\pi(y,k,k)\in Y'\text{ for some }k=0,1,\dots,f(y)\}. 
\]
Since $Y'$ is $G$-full, we can see that $Z$ is $G$-full. 
For each $z\in Z$, we let 
\[
g(z)=\min\{k\in\{0,1,\dots,f(z)\}\mid\pi(z,k,k)\in Y'\}
\]
and 
$U=\{\pi(z,g(z),0)\mid z\in Z\}$. 
Then $g$ is a continuous function on $Z$ and 
$U$ is a compact open $G$-set satisfying $s(U)=Z$ and $r(U)\subset Y'$. 
Clearly $Z'=r(U)$ is $G$-full, and $G|Z$ and $G|Z'$ are isomorphic. 
Hence $G|Y$ and $G|Y'$ are Kakutani equivalent. 
\end{proof}

\begin{lem}\label{Kakutaniequiv}
The Kakutani equivalence is an equivalence relation 
between \'etale groupoids 
whose unit spaces are compact and totally disconnected. 
\end{lem}
\begin{proof}
It suffices to prove transitivity. 
Let $G_i$ be an \'etale groupoid 
whose unit space is compact and totally disconnected for $i=1,2,3$. 
Suppose that $G_1$ and $G_2$ are Kakutani equivalent and that 
$G_2$ and $G_3$ are Kakutani equivalent. 
We can find clopen subsets $Y_1\subset G_1^{(0)}$, 
$Y_2,Y'_2\subset G_2^{(0)}$ and $Y_3\subset G_3^{(0)}$ such that 
each of them are full, 
$G_1|Y_1$ is isomorphic to $G_2|Y_2$ and 
$G_2|Y'_2$ is isomorphic to $G_3|Y_3$. 
Let $\pi:G_2|Y_2\to G_1|Y_1$ and $\pi':G_2|Y'_2\to G_3|Y_3$ be isomorphisms. 
From Lemma \ref{tower2}, 
there exist $G_2$-full clopen subsets $Z\subset Y_2$ and $Z'\subset Y'_2$ 
such that $G_2|Z$ is isomorphic to $G_2|Z'$. 
Then $G_1|\pi(Z)$ is isomorphic to $G_3|\pi'(Z')$, and so 
$G_1$ and $G_3$ are Kakutani equivalent. 
\end{proof}

\begin{lem}\label{Kakutani=tower}
Let $G_i$ be an \'etale groupoid 
whose unit space is compact and totally disconnected for $i=1,2$. 
The following are equivalent. 
\begin{enumerate}
\item $G_1$ is Kakutani equivalent to $G_2$. 
\item There exist $f_i\in C(G_i^{(0)},\Z)$ such that 
$(G_1)_{f_1}$ is isomorphic to $(G_2)_{f_2}$. 
\end{enumerate}
\end{lem}
\begin{proof}
(1)$\Rightarrow$(2). 
By Lemma \ref{tower1}, 
we can find an \'etale groupoid $G$ and $g_1,g_2\in C(G^{(0)},\Z)$
such that $G_{g_i}$ is isomorphic to $G_i$. 
Let $\pi:G_i\to G_{g_i}$ be the isomorphism. 
Put $g(x)=\max\{g_1(x),g_2(x)\}$. 
Define $f_i\in C(G_{g_i}^{(0)},\Z)$ by 
\[
f_i(x,k,k)=\begin{cases}
g(x)-g_i(x) & \text{if }k=0 \\
0 & \text{otherwise}. \end{cases}
\]
Let $h_i=f_i\circ(\pi|G_i^{(0)})$. 
It is easy to see that 
$(G_i)_{h_i}$ is isomorphic to $(G_{g_i})_{f_i}$ and 
that $(G_{g_i})_{f_i}$ is isomorphic to $G_g$. 
Therefore we get (2). 

(2)$\Rightarrow$(1). 
For $i=1,2$, 
we let $Y_i=\{(x,0,0)\in(G_i)_{f_i}^{(0)}\mid x\in G_i^{(0)}\}$. 
Then $Y_i$ is $(G_i)_{f_i}$-full and 
$(G_i)_{f_i}|Y_i$ is isomorphic to $G_i$. 
It follows from Lemma \ref{tower2} that 
$G_1$ and $G_2$ are Kakutani equivalent. 
\end{proof}

From the lemma above, one can see that 
Kakutani equivalence is a generalization of 
bounded orbit injection equivalence introduced in \cite[Definition 1.3]{LaO} 
(see also Definition 5 and Theorem 6 of \cite{LiO}). 

\begin{lem}\label{cptKakutani}
Let $G$ be an \'etale groupoid 
whose unit space is compact and totally disconnected. 
The following are equivalent. 
\begin{enumerate}
\item $G$ is principal and compact. 
\item $G$ is Kakutani equivalent to $H$ such that $H=H^{(0)}$. 
\end{enumerate}
\end{lem}
\begin{proof}
This is immediate from \cite[Lemma 3.4]{GPS3} and 
the definition of Kakutani equivalence. 
\end{proof}

\subsection{Examples of homology groups}

Next, we turn to the consideration of homology of 
an \'etale groupoid $G$ 
whose unit space is compact and totally disconnected. 

\begin{thm}\label{Kakutani>sim}
Let $G_i$ be an \'etale groupoid 
whose unit space is compact and totally disconnected for $i=1,2$. 
If $G_1$ and $G_2$ are Kakutani equivalent, 
then $G_1$ is homologically similar to $G_2$. 
In particular, $H_n(G_1,A)$ is isomorphic to $H_n(G_2,A)$ 
for any topological abelian group $A$. 
Moreover, there exists an isomorphism $\pi:H_0(G_1)\to H_0(G_2)$ 
such that $\pi(H_0(G_1)^+)=H_0(G_2)^+$. 
\end{thm}
\begin{proof}
This follows from the definition of Kakutani equivalence, 
Theorem \ref{reduction} and Proposition \ref{homosim}. 
\end{proof}

\begin{lem}\label{cptHomology}
Let $G$ be a compact \'etale principal groupoid 
whose unit space is compact and totally disconnected. 
Let $A$ be a topological abelian group. 
Then $H_n(G,A)=0$ for $n\geq1$. 
\end{lem}
\begin{proof}
This follows from Lemma \ref{cptKakutani}, Theorem \ref{Kakutani>sim} 
and the definition of homology groups. 
\end{proof}

As for AF groupoids (i.e. AF equivalence relations), 
we have the following. 
For the definition of dimension groups, 
we refer the reader to \cite[Section 1.4]{Ro}. 

\begin{thm}\label{AFHomology1}
\begin{enumerate}
\item For an AF groupoid $G$, 
there exists an isomorphism $\pi:H_0(G)\to K_0(C^*_r(G))$ such that 
$\pi(H_0(G)^+)=K_0(C^*_r(G))^+$ and $\pi([1_{G^{(0)}}])=[1_{C^*_r(G)}]$. 
In particular, the triple $(H_0(G),H_0(G)^+,[1_{G^{(0)}}])$ is 
a dimension group with a distinguished order unit. 
\item Two AF groupoids $G_1$ and $G_2$ are isomorphic if and only if 
there exists an isomorphism $\pi:H_0(G_1)\to H_0(G_2)$ such that 
$\pi(H_0(G_1)^+)=H_0(G_2)^+$ and $\pi([1_{G_1^{(0)}}])=[1_{G_2^{(0)}}]$. 
\end{enumerate}
\end{thm}
\begin{proof}
The first statement follows from \cite{R1,HPS,GPS1,GPS3}. 
The second statement was proved in \cite{Kr}. 
\end{proof}

\begin{thm}\label{AFHomology2}
Let $G$ be an AF groupoid and 
let $A$ be a topological abelian group. 
Then $H_n(G,A)=0$ for $n\geq1$. 
\end{thm}
\begin{proof}
The AF groupoid $G$ is an increasing union of elementary subgroupoids. 
For any $f\in C_c(G^{(n)},A)$, 
there exists an elementary subgroupoid $K\subset G$ such that 
$f\in C_c(K^{(n)},A)$. 
By Lemma \ref{cptHomology}, $H_n(K,A)=0$ for $n\geq1$. 
Therefore $H_n(G,A)=0$ for $n\geq1$. 
\end{proof}

\begin{rem}
Let $G$ be an AF groupoid and 
let $A$ be a topological abelian group. 
It is known that 
the cohomology group $H^n(G,A)$ with constant coefficients is zero 
for every $n\geq2$ (\cite[III.1.3]{R1}). 
Clearly $H^0(G,A)$ is equal to 
the set of continuous functions $f\in C(G^{(0)},A)$ satisfying 
$f(r(g))=f(s(g))$ for all $g\in G$. 
In particular, when $G$ is minimal, $H^0(G,A)$ is isomorphic to $A$. 
When $G$ is minimal and $A=\Z$, 
one can see that $H^1(G,\Z)$ is always uncountable. 
\end{rem}

We now turn to a computation of homology groups 
of \'etale groupoids arising from subshifts of finite type. 
We refer the reader to \cite{A,R2} 
for more details about these groupoids. 
Let $\sigma$ be a one-sided subshift of finite type 
on a compact totally disconnected space $X$. 
We assume that $\sigma$ is surjective. 
The \'etale groupoid $G$ associated with $\sigma$ is given by 
\[
G=\{(x,n,y)\in X\times\Z\times X\mid
\exists k,l\in\N,\ n=k{-}l,\ \sigma^k(x)=\sigma^l(y)\}. 
\]
Two elements $(x,n,y)$ and $(x',n',y')$ in $G$ are composable 
if and only if $y=x'$, and the multiplication and the inverse are 
\[
(x,n,y)\cdot(y,n',y')=(x,n{+}n',y'),\quad (x,n,y)^{-1}=(y,-n,x). 
\]
The \'etale groupoid $G$ has 
an open subgroupoid $H=\{(x,0,y)\in G\}$. 
It is well-known that 
$C^*_r(G)$ is isomorphic to the Cuntz-Krieger algebra 
introduced in \cite{CK} 
and that $H$ is an AF groupoid. 
Moreover, there exists an automorphism $\pi$ of $K_0(C^*_r(H))$ such that 
$K_0(C^*_r(G))\cong\Coker(\id-\pi)$ and $K_1(C^*_r(G))\cong\ker(\id-\pi)$. 
It is also well-known that 
$K_0(C^*_r(G))$ is a finitely generated abelian group and 
$K_1(C^*_r(G))$ is isomorphic to the torsion-free part of $K_0(C^*_r(G))$. 

The map $\rho:(x,n,y)\mapsto n$ is a homomorphism from $G$ to $\Z$. 
We consider the skew product $G\times_\rho\Z$ and 
set $Y=G^{(0)}\times\{0\}\subset(G\times_\rho\Z)^{(0)}$. 

\begin{lem}
In the setting above, 
$G\times_\rho\Z$ is homologically similar to $H$. 
\end{lem}
\begin{proof}
It is easy to see that $Y$ is $(G\times_\rho\Z)$-full. 
In addition, $(G\times_\rho\Z)|Y$ is canonically isomorphic to $H$. 
By Theorem \ref{reduction} (2), we get the conclusion. 
\end{proof}

\begin{thm}\label{CKHomology}
When $G$ is the \'etale groupoid arising from a subshift of finite type, 
$H_0(G)\cong K_0(C^*_r(G))$, $H_1(G)\cong K_1(C^*_r(G))$ and 
$H_n(G)=0$ for all $n\geq2$. 
\end{thm}
\begin{proof}
It follows from Theorem \ref{LHS} (1) that 
there exists a spectral sequence: 
\[
E^2_{p,q}=H_p(\Z,H_q(G\times_\rho\Z))\Rightarrow
H_{p+q}(G). 
\]
By the lemma above and Proposition \ref{homosim}, 
$H_q(G\times_\rho\Z)$ is isomorphic to $H_q(H)$. 
This, together with Theorem \ref{AFHomology1} and Theorem \ref{AFHomology2}, 
implies 
\[
H_q(G\times_\rho\Z)\cong\begin{cases}K_0(C^*_r(H)) & q=0 \\
0 & q\geq1. \end{cases}
\]
Besides, 
the $\Z$-module structure on $H_0(G\times_\rho\Z)\cong K_0(C^*_r(H))$ is 
given by the automorphism $\pi$. 
Hence one has 
\[
H_0(G)\cong H_0(\Z,K_0(C^*_r(H)))\cong\Coker(\id-\pi)\cong K_0(C^*_r(G)), 
\]
\[
H_1(G)\cong H_1(\Z,K_0(C^*_r(H)))\cong\ker(\id-\pi)\cong K_1(C^*_r(G))
\]
and $H_n(G)=0$ for $n\geq2$. 
\end{proof}

\section{Kakutani equivalence and $C^*$-algebras}

In this section, 
we give a $C^*$-algebraic characterization of Kakutani equivalence. 
For an \'etale groupoid $G$, 
we denote the reduced groupoid $C^*$-algebra of $G$ by $C^*_r(G)$ and 
identify $C_0(G^{(0)})$ with a subalgebra of $C^*_r(G)$. 
The following is an immediate consequence of (a special case of) 
Proposition 4.11 of \cite{R3}. 

\begin{thm}\label{isomC*}
For $i=1,2$, 
let $G_i$ be an \'etale essentially principal second countable groupoid. 
The following are equivalent. 
\begin{enumerate}
\item $G_1$ and $G_2$ are isomorphic. 
\item There exists an isomorphism $\pi:C^*_r(G_1)\to C^*_r(G_2)$ 
such that $\pi(C_0(G_1^{(0)}))=C_0(G_2^{(0)})$. 
\end{enumerate}
\end{thm}

The following lemma is obvious from the definition of $C^*_r(G)$. 

\begin{lem}
Let $G$ be an \'etale groupoid whose unit space is compact and 
let $Y\subset G^{(0)}$ be a clopen subset. 
There exists a natural isomorphism $\pi:C^*_r(G|Y)\to1_YC^*_r(G)1_Y$ 
such that $\pi(f)=f$ for every $f\in C(Y)$. 
\end{lem}

An element in a $C^*$-algebra $A$ is said to be full 
if it is not contained in any proper closed two-sided ideal in $A$.
The following is an easy consequence of \cite[Proposition II.4.5]{R1}. 

\begin{lem}
Let $G$ be an \'etale groupoid whose unit space is compact and 
let $Y\subset G^{(0)}$ be a clopen subset. 
Then, $Y$ is $G$-full if and only if 
$1_Y$ is a full projection in $C^*_r(G)$. 
\end{lem}

Combining Theorem \ref{isomC*} with the above two lemmas, 
we get the following. 
This is a generalization of \cite[Theorem 2.4]{LaO}. 

\begin{thm}\label{C*Kakutani}
Let $G_i$ be an \'etale essentially principal second countable groupoid 
whose unit space is compact and totally disconnected for $i=1,2$. 
The following are equivalent. 
\begin{enumerate}
\item $G_1$ and $G_2$ are Kakutani equivalent. 
\item There exist two projections $p_1\in C(G_1^{(0)})$, $p_2\in C(G_2^{(0)})$ 
and an isomorphism $\pi$ from $p_1C^*_r(G_1)p_1$ to $p_2C^*_r(G_2)p_2$ 
such that $p_i$ is full in $C^*_r(G_i)$ and 
$\pi(p_1C(G_1^{(0)}))=p_2C(G_2^{(0)})$. 
\end{enumerate}
\end{thm}

We next consider relationship 
between $[[G]]$ and unitary normalizers of $C(G^{(0)})$ in $C^*_r(G)$. 
In what follows, 
an element in $C^*_r(G)$ is identified with a function in $C_0(G)$ 
(\cite[II.4.2]{R1}). 
The following is a slight generalization of \cite[II.4.10]{R1} and 
a special case of \cite[Proposition 4.7]{R3}. 
We omit the proof. 

\begin{lem}\label{partialisom}
Let $G$ be an \'etale essentially principal second countable groupoid. 
Suppose that $v\in C^*_r(G)$ is a partial isometry satisfying 
$v^*v,vv^*\in C_0(G^{(0)})$ and $vC_0(G^{(0)})v^*=vv^*C_0(G^{(0)})$. 
Then there exists a compact open $G$-set $V\subset G$ such that 
\[
V=\{g\in G\mid v(g)\neq0\}=\{g\in G\mid\lvert v(g)\rvert=1\}. 
\]
In addition, for any $f\in v^*vC_0(G^{(0)})$, 
one has $vfv^*=f\circ\tau_{V}^{-1}$. 
\end{lem}

For a unital $C^*$-algebra $A$, 
let $U(A)$ denote the unitary group of $A$ 
and for a subalgebra $B\subset A$, 
let $N(B,A)$ denote the normalizer of $B$ in $U(A)$, that is, 
\[
N(B,A)=\{u\in U(A)\mid uBu^*=B\}. 
\]
Let $G$ be an \'etale groupoid 
whose unit space $G^{(0)}$ is compact. 
Clearly $U(C(G^{(0)}))$ is a subgroup of $N(C(G^{(0)}),C^*_r(G))$ 
and we let $\iota$ denote the inclusion map. 
An element $u\in N(C(G^{(0)}),C^*_r(G))$ induces an automorphism 
$f\mapsto ufu^*$ of $C(G^{(0)})$, and so 
there exists a homomorphism 
$\sigma:N(C(G^{(0)}),C^*_r(R))\to\Homeo(G^{(0)})$ such that 
$ufu^*=f\circ\sigma(u)^{-1}$ for all $f\in C(G^{(0)})$. 
The following is a generalization of 
\cite[Section 5]{P}, \cite[Theorem 1]{T} and \cite[Theorem 1.2]{Matsu}. 

\begin{prop}\label{exact}
Suppose that $G$ is an \'etale essentially principal second countable groupoid 
whose unit space is compact. 
\begin{enumerate}
\item The image of $\sigma$ is contained 
in the topological full group $[[G]]$. 
\item The sequence 
\[
1\longrightarrow U(C(G^{(0)}))\stackrel{\iota}{\longrightarrow}
N(C(G^{(0)}),C^*_r(G))\stackrel{\sigma}{\longrightarrow}
[[G]]\longrightarrow1
\]
is exact. 
\item The homomorphism $\sigma$ has a right inverse. 
\end{enumerate}
\end{prop}
\begin{proof}
(1) This follows from Lemma \ref{partialisom} and 
the definition of $[[G]]$. 

(2) By definition, $\iota$ is injective. 
Since $C(G^{(0)})$ is abelian, 
the image of $\iota$ is contained in the kernel of $\sigma$. 
By \cite[Proposition II.4.7]{R1}, 
$C(G^{(0)})$ is a maximal abelian subalgebra. 
It follows that the kernel of $\sigma$ is contained in the image of $\iota$. 
The surjectivity of $\sigma$ follows from (3). 

(3) Take $\gamma\in[[G]]$. 
Since $G$ is essentially principal, 
there exists a unique compact open $G$-set $U\subset G$ 
such that $\gamma=\tau_U$. 
We let $u\in C_c(G)$ be the characteristic function of $U$. 
One can see that $u$ is a unitary in $C^*_r(G)$ 
satisfying $ufu^*=f\circ\gamma^{-1}$ for all $f\in C(G^{(0)})$. 
The map $\gamma\mapsto u$ gives a right inverse of $\sigma$. 
\end{proof}

Let $G$ be an \'etale essentially principal groupoid 
whose unit space is compact. 
We let $\Aut_{C(G^{(0)})}(C^*_r(G))$ denote 
the group of automorphisms of $C^*_r(G)$ preserving $C(G^{(0)})$ globally. 
Thus 
\[
\Aut_{C(G^{(0)})}(C^*_r(G))
=\{\alpha\in\Aut(C^*_r(G))\mid\alpha(C(G^{(0)}))=C(G^{(0)})\}. 
\]
We let $\Inn_{C(G^{(0)})}(C^*_r(G))$ denote the subgroup 
of $\Aut_{C(G^{(0)})}(C^*_r(G))$ consisting of inner automorphisms. 
In other words, 
\[\Inn_{C(G^{(0)})}(C^*_r(G))
=\{\Ad u\mid u\in N(C(G^{(0)}),C^*_r(G))\}. 
\]
Let $\Out_{C(G^{(0)})}(C^*_r(G))$ be the quotient group of 
$\Aut_{C(G^{(0)})}(C^*_r(G))$ by $\Inn_{C(G^{(0)})}(C^*_r(G))$. 
The automorphism group of $G$ is denoted by $\Aut(G)$. 
For $\gamma=\tau_O\in[[G]]$, 
there exists an automorphism $\phi_\gamma\in\Aut(G)$ satisfying 
\[
OgO^{-1}=\{\phi_\gamma(g)\}
\]
for all $g\in G$. 
We regard $[[G]]$ as a subgroup of $\Aut(G)$ 
via the identification of $\gamma$ with $\phi_\gamma$. 
Thanks to \cite[Proposition 4.11]{R3}, 
we can prove the following exact sequences for these automorphism groups, 
which generalize \cite[Proposition 2.4]{GPS2}, \cite[Theorem 3]{T2} and 
\cite[Theorem 1.3]{Matsu}. 
See \cite[Definition I.1.12]{R1} for the definition of 
$Z^1(G,\T)$, $B^1(G,\T)$ and $H^1(G,\T)$. 

\begin{prop}\label{exact2}
Let $G$ be an \'etale essentially principal second countable groupoid 
whose unit space is compact. 
\begin{enumerate}
\item There exist short exact sequences: 
\[
1\longrightarrow Z^1(G,\T)\stackrel{j}{\longrightarrow}
\Aut_{C(G^{(0)})}(C^*_r(G))\stackrel{\omega}{\longrightarrow}
\Aut(G)\longrightarrow1, 
\]
\[
1\longrightarrow B^1(G,\T)\stackrel{j}{\longrightarrow}
\Inn_{C(G^{(0)})}(C^*_r(G))\stackrel{\omega}{\longrightarrow}
[[G]]\longrightarrow1, 
\]
\[
1\longrightarrow H^1(G,\T)\stackrel{j}{\longrightarrow}
\Out_{C(G^{(0)})}(C^*_r(G))\stackrel{\omega}{\longrightarrow}
\Aut(G)/[[G]]\longrightarrow1. 
\]
Moreover, they all split, that is, $\omega$ has a right inverse. 
\item 
Suppose that $G$ admits a covering by compact open $G$-sets $O$ 
satisfying $r(O)=s(O)=G^{(0)}$. 
Then $\Aut(G)$ is naturally isomorphic to 
the normalizer $N([[G]])$ of $[[G]]$ in $\Homeo(G^{(0)})$. 
\end{enumerate}
\end{prop}
\begin{proof}
(1)
It suffices to show the exactness of the first sequence, 
because the others are immediately obtained from the first one. 
Take $\alpha\in\Aut_{C(G^{(0)})}(C^*_r(G))$. 
Clearly $\alpha$ induces an automorphism of the Weyl groupoid of 
$(C^*_r(G),C(G^{(0)}))$ (\cite[Definition 4.2]{R3}), 
which is canonically isomorphic to $G$ by \cite[Proposition 4.11]{R3}. 
Therefore there exists a homomorphism $\omega$ 
from $\Aut_{C(G^{(0)})}(C^*_r(G))$ to $\Aut(G)$. 
Evidently $\omega$ is surjective and has a right inverse. 
Take $\xi\in Z^1(G,\T)$. 
We can define $j(\xi)\in\Aut_{C(G^{(0)})}(C^*_r(G))$ by setting 
\[
j(\xi)(f)(g)=\xi(g)f(g)
\]
for $f\in C^*_r(G)$ and $g\in G$. 
Obviously $j$ is an injective homomorphism and 
$\Ima j$ is contained in $\ker\omega$. 
It remains for us to show that $\ker\omega$ is contained in $\Ima j$. 
For $f\in C_c(G)$, we write $\supp(f)=\{g\in G\mid f(g)\neq0\}$. 
Suppose $\alpha\in\ker\omega$. 
We have $\alpha(f)=f$ for $f\in C(G^{(0)})$. 
Take $g\in G$. 
Choose $u\in C_c(G)$ so that $u(g)>0$, $u(h)\geq0$ for all $h\in G$ and 
$\supp(u)$ is an open $G$-set. 
Then $\alpha(u)(g)/u(g)$ is in $\T$, because
\[
\lvert u(g)\rvert^2=\lvert u^*u(s(g))\rvert
=\lvert \alpha(u^*u)(s(g))\rvert=\lvert \alpha(u)(g)\rvert^2. 
\]
Let $v\in C_c(G)$ be a function which has the same properties as $u$. 
Then $O=\supp(u)\cap\supp(v)$ is an open $G$-set containing $g$. 
Let $w\in C(G^{(0)})$ be a positive element 
satisfying $w(s(g))>0$ and $\supp(w)\subset s(O)$. 
We have 
\[
(u(v^*v)^{1/2}w)(h)=u(h)((v^*v)^{1/2}w)(s(h))=u(h)v(h)w(s(h))
\]
for every $h\in G$. 
Since we also have the same equation for $v(u^*u)^{1/2}w$, 
we can conclude $u(v^*v)^{1/2}w=v(u^*u)^{1/2}w$. 
Accordingly, one obtains 
\begin{align*}
\alpha(u)(g)v(g)w(s(g))
&=(\alpha(u)(v^*v)^{1/2}w)(g) \\
&=\alpha(u(v^*v)^{1/2}w)(g) \\
&=\alpha(v(u^*u)^{1/2}w)(g) \\
&=(\alpha(v)(u^*u)^{1/2}w)(g)=\alpha(v)(g)u(g)w(s(g)). 
\end{align*}
It follows that 
the value $\alpha(u)(g)/u(g)$ does not depend on the choice of $u$. 
We write it by $\xi(g)$. 
From the definition, 
it is easy to verify that $\xi$ belongs to $Z^1(G,\T)$ and 
that $\alpha$ is equal to $j(\xi)$. 

(2)
For $\gamma\in\Aut(G)$, it is easy to see that 
$\gamma|G^{(0)}$ is in the normalizer $N([[G]])$ of $[[G]]$. 
The map $q:\Aut(G)\to N([[G]])$ sending $\gamma$ to $\gamma|G^{(0)}$ 
is clearly a homomorphism. 
Since $G$ is essentially principal, one can see that $q$ is injective. 
Let $h\in N([[G]])$. 
For $g\in G$, let $O$ be a compact open $G$-set such that 
$r(O)=s(O)=G^{(0)}$ and $g\in O$. 
There exists a compact open $G$-set $O'$ 
such that $h\circ\tau_O\circ h^{-1}=\tau_{O'}$, 
because $h$ is in the normalizer of $[[G]]$. 
Set $g'=(r|O')^{-1}(h(r(g))$. 
Clearly $r(g')=h(r(g))$ and $s(g')=h(s(g))$. 
As $G$ is essentially principal, 
we can conclude that $g'$ does not depend on the choice of $O$. 
We define $\gamma:G\to G$ by letting $\gamma(g)=g'$. 
It is not so hard to see that $\gamma$ is in $\Aut(G)$ and $q(\gamma)=h$, 
which means that $q$ is surjective. 
Consequently, $q$ is an isomorphism. 
\end{proof}

\begin{rem}
Any transformation groupoid $G_\phi$ and 
any principal and totally disconnected $G$ satisfy 
the hypothesis of Proposition \ref{exact2} (2). 
When $G$ is an \'etale groupoid arising from a subshift of finite type, 
$\Aut(G)$ is naturally isomorphic to $N([[G]])$ 
(\cite[Theorem 1.3]{Matsu}). 
\end{rem}

\section{Almost finite groupoids}

In this section we introduce the notion of almost finite groupoids 
(Definition \ref{almstfnt}). 
Transformation groupoids arising from free actions of $\Z^N$ are 
shown to be almost finite (Lemma \ref{ZNisAF}). 
Moreover, for two $G$-full clopen subsets, 
we prove that they have the same class in $H_0(G)$ 
if and only if there exists an element in $[[G]]$ which maps one to the other. 

Throughout this section, 
we let $G$ be a second countable \'etale groupoid 
whose unit space is compact and totally disconnected. 
The equivalence class of $f\in C(G^{(0)},\Z)$ in $H_0(G)$ is denoted by $[f]$. 
A probability measure $\mu$ on $G^{(0)}$ is said to be $G$-invariant 
if $\mu(r(U))=\mu(s(U))$ holds for every open $G$-set $U$. 
The set of all $G$-invariant measures is denoted by $M(G)$. 
For $\mu\in M(G)$, 
we can define a homomorphism $\hat\mu:H_0(G)\to\R$ by 
\[
\hat\mu([f])=\int f\,d\mu. 
\]

\subsection{Almost finite groupoids}

The following lemma will be used repeatedly later. 

\begin{lem}\label{cptHopf}
Suppose that $G$ is compact and principal. 
\begin{enumerate}
\item If a clopen subset $U\subset G^{(0)}$ and $c>0$ satisfy 
$\lvert G(x)\cap U\rvert<c\lvert G(x)\rvert$ for all $x\in G^{(0)}$, then 
$\mu(U)<c$ for all $\mu\in M(G)$. 
\item Let $U_1,U_2,\dots,U_n$ and $O$ be clopen subsets of $G^{(0)}$ 
satisfying $\sum_{i=1}^n\lvert G(x)\cap U_i\rvert\leq\lvert G(x)\cap O\rvert$ 
for any $x\in G^{(0)}$. 
Then there exist compact open $G$-sets $C_1,C_2,\dots,C_n$ such that 
$r(C_i)=U_i$, $s(C_i)\subset O$ for all $i$ and 
$s(C_i)$'s are mutually disjoint. 
\item Let $U$ and $V$ be clopen subsets of $G^{(0)}$ 
satisfying $\lvert G(x)\cap U\rvert=\lvert G(x)\cap V\rvert$ 
for any $x\in G^{(0)}$. 
Then there exists a compact open $G$-set $C$ such that 
$r(C)=U$ and $s(C)=V$. 
\item For $f\in C(G^{(0)},\Z)$, 
$[f]$ is in $H_0(G)^+$ if and only if 
$\hat\mu([f])\geq0$ for every $\mu\in M(G)$. 
\end{enumerate}
\end{lem}
\begin{proof}
(1) is clear from the definition. 
(3) easily follows from (2). 
(4) can be proved in a similar fashion to (2). 
We show only (2). 
By Lemma \ref{cptKakutani}, 
there exists a $G$-full clopen subset $Y\subset G^{(0)}$ 
such that $G|Y=Y$. 
It follows from Lemma \ref{tower1} that 
there exist $f\in C(Y,\Z)$ and an isomorphism $\pi$ from $(G|Y)_f=Y_f$ to $G$ 
such that $\pi(y,0,0)=y$ for all $y\in Y$. 
For each $(n{+}1)$-tuple $\Lambda=(\lambda_0,\lambda_1,\dots,\lambda_n)$ 
of finite subsets of $\N$ satisfying 
\[
\lvert\lambda_0\rvert
\geq\lvert\lambda_1\rvert+\lvert\lambda_2\rvert+\dots+\lvert\lambda_n\rvert, 
\]
we fix an $n$-tuple $\alpha_\Lambda=(\alpha_1,\alpha_2,\dots,\alpha_n)$ 
of injective maps $\alpha_i:\lambda_i\to\lambda_0$ 
such that $\alpha_i(\lambda_i)\cap\alpha_j(\lambda_j)=\emptyset$ 
for $i\neq j$. 
For $y\in Y$, we set 
\[
\lambda_0(y)=\{k\mid\pi(y,k,k)\in O\},\quad 
\lambda_i(y)=\{k\mid\pi(y,k,k)\in U_i\}\quad
\forall i=1,2,\dots,n
\]
and $\Lambda(y)=(\lambda_0(y),\lambda_1(y),\dots,\lambda_n(y))$. 
Let $\alpha_{y,i}$ be the $i$-th summand of $\alpha_{\Lambda(y)}$. 
For $i=1,2,\dots,n$, we define $C_i\subset G$ by 
\[
\pi^{-1}(C_i)=\{(y,k,l)\in(G|Y)_f\mid
k\in\lambda_i(y),\ \alpha_{y,i}(k)=l\}. 
\]
Then one can verify that 
$C_i$'s are compact open $G$-sets which meet the requirement. 
\end{proof}

\begin{df}\label{almstfnt}
We say that $G$ is almost finite, 
if for any compact subset $C\subset G$ and $\ep>0$ 
there exists an elementary subgroupoid $K\subset G$ such that 
\[
\frac{\lvert CKx\setminus Kx\rvert}{\lvert K(x)\rvert}<\ep
\]
for all $x\in G^{(0)}$, 
where $K(x)$ stands for the $K$-orbit of $x$. 
We also remark that $\lvert K(x)\rvert$ equals $\lvert Kx\rvert$, 
because $K$ is principal. 
\end{df}

This definition may remind the reader 
of the F\o lner condition for amenable groups. 
While there is no direct relationship between them, 
it may be natural to expect that 
transformation groupoids arising from free actions of amenable groups 
are almost finite. 
Indeed, the next lemma shows that this is true at least for $\Z^N$. 
Notice that $\phi$ need not be minimal in the following statement. 

\begin{lem}\label{ZNisAF}
When $\phi:\Z^N\curvearrowright X$ is a free action of $\Z^N$ 
on a compact, metrizable and totally disconnected space $X$, 
the transformation groupoid $G_\phi$ is almost finite. 
\end{lem}
\begin{proof}
We follow the arguments in \cite{F} (see also \cite{GMPS,GMPS2}). 
We regard $\Z^N$ as a subset of $\R^N$ and 
let $\lVert\cdot\rVert$ denote the Euclid norm on $\R^N$. 
Let $\omega$ be the volume of the closed unit ball of $\R^N$. 
We also equip $\Z^N$ with the lexicographic order. 
Namely, for $p=(p_1,p_2,\dots,p_N)$ and $q=(q_1,q_2,\dots,q_N)$ in $\Z^N$, 
$p$ is less than $q$ if there exists $i$ such that 
$p_i<q_i$ and $p_j=q_j$ for all $j<i$. 

Suppose that a compact subset $C\subset G_\phi$ and $\ep>0$ are given. 
There exists $n\in\N$ such that 
$C$ is contained in $\{(p,x)\in G_\phi\mid\lVert p\rVert\leq n,\ x\in X\}$. 
We identify $x\in X$ with $(0,x)\in G_\phi$. 
Choose $m$ sufficiently large. 
By \cite[Lemma 20]{LiO} or \cite[Proposition 4.4]{GMPS}, 
we can construct a clopen subset $U\subset X$ such that 
$\bigcup_{\lVert p\rVert\leq m}\phi^p(U)=X$ and 
$U\cap\phi^p(U)=\emptyset$ for any $p$ with $0<\lVert p\rVert\leq m$. 
For each $x\in X$, we let 
\[
P(x)=\{p\in\Z^N\mid\phi^p(x)\in U\}. 
\]
Then $P(x)\subset\R^N$ is $m$-separated and $(m{+}1)$-syndetic 
in the sense of \cite{GMPS}. 
Let $f(x)\in\Z^N$ be the minimum element of 
\[
\{p\in P(x)\mid\lVert p\rVert\leq\lVert q\rVert\quad\forall q\in P(x)\}
\]
in the lexicographic order. 
Define $K\subset G_\phi$ by 
\[
K=\{(p,x)\in G_\phi\mid f(x)=p+f(\phi^p(x))\}. 
\]
Then $K$ is an elementary subgroupoid of $G$ (see \cite{F,GMPS}). 

We would like to show that $K$ meets the requirement. 
Fix $x\in U$ and 
consider the Voronoi tessellation with respect to $P(x)$. 
Let $T$ be the Voronoi cell containing the origin, that is, 
\[
T=\{q\in\R^N\mid
\lVert q\rVert\leq\lVert q-p\rVert\quad\forall p\in P(x)\}. 
\]
Then $T$ is a convex polytope. 
Note that if $q\in\Z^N$ is in the interior of $T$, then $(q,x)$ is in $K$. 
Since $P(x)$ is $m$-separated, 
$T$ contains the closed ball of radius $m/2$ centred at the origin. 
Hence $\lvert K(x)\rvert\geq(m/2-2)^N$. 
On the other hand, 
$T$ is contained in the closed ball of radius $m{+}1$ centred at the origin, 
because $P(x)$ is $(m{+}1)$-syndetic. 
It follows that 
the volume $V$ of $T$ is not greater than $(m{+}1)^N\omega$. 
Let $B_1$ be the set of $q\in T$ 
which is within distance $n{+}1$ from the boundary of $T$ and 
let $B_2$ be the set of $q\in\R^N\setminus T$ 
which is within distance $1$ from the boundary of $T$. 
If $q\in T\cap\Z^N$ is within distance $n$ from the boundary of $T$, 
then the closed ball of radius $1/2$ centered at $q$ is 
contained in $B_1\cup B_2$. 
Since $B_1\cup B_2$ is contained in 
\[
\left\{\theta q\in\R^N\mid1-\frac{2(n+1)}{m}\leq\theta\leq1+\frac{2}{m},\
q\in T\right\}, 
\]
the volume of $B_1\cup B_2$ is less than 
\[
\left\{\left(1+\frac{2}{m}\right)^N-\left(1-\frac{2(n+1)}{m}\right)^N\right\}V
\leq\frac{((m+2)^N-(m-2n-2)^N)(m+1)^N\omega}{m^N}. 
\]
Hence 
\[
\lvert CKx\setminus Kx\rvert
\leq\frac{((m+2)^N-(m-2n-2)^N)(m+1)^N\omega}{m^N}\times((1/2)^N\omega)^{-1}. 
\]
As a consequence, by choosing $m$ sufficiently large, 
we get 
\[
\frac{\lvert CKx\setminus Kx\rvert}{\lvert K(x)\rvert}<\ep. 
\]
\end{proof}

\begin{rem}
Let $\phi:\R^N\curvearrowright\Omega$ be a free action of $\R^N$ 
on a compact, metrizable space $\Omega$ and 
let $X\subset\Omega$ be a flat Cantor transversal 
in the sense of \cite[Definition 2.1]{GMPS2}. 
As described in \cite{GMPS2}, 
we can construct an \'etale principal groupoid 
(i.e. \'etale equivalence relation) whose unit space is (homeomorphic to) $X$. 
In the same way as the lemma above, 
this groupoid is shown to be almost finite. 
\end{rem}

\subsection{Basic facts about almost finite groupoids}

In this subsection 
we collect several basic facts about almost finite groupoids. 

\begin{lem}\label{Mneqempty}
If $G$ is almost finite, then $M(G)$ is not empty. 
\end{lem}
\begin{proof}
Take an increasing sequence of compact open subsets 
$C_1\subset C_2\subset\dots$ whose union is equal to $G$. 
For each $n\in\N$, 
there exists an elementary subgroupoid $K_n\subset G$ such that 
\[
\frac{\lvert C_nK_nx\setminus K_nx\rvert}{\lvert K_n(x)\rvert}<\frac{1}{n}
\]
for all $x\in G^{(0)}$. 
Clearly $M(K_n)$ is not empty, and so we can choose $\mu_n\in M(K_n)$. 
By taking a subsequence if necessary, 
we may assume that $\mu_n$ converges to a probability measure $\mu$. 
We would like to show that $\mu$ belongs to $M(G)$. 

Let $U$ be a compact open $G$-set. 
For sufficiently large $n\in\N$, $C_n$ contains $U\cup U^{-1}$. 
For $x\in G^{(0)}$, one has 
\[
\lvert K_n(x)\cap s(U\setminus K_n)\rvert
=\lvert (U\setminus K_n)K_nx\rvert=\lvert UK_nx\setminus K_nx\rvert, 
\]
which is less than $n^{-1}\lvert K_n(x)\rvert$ 
when $C_n$ contains $U$. 
It follows from Lemma \ref{cptHopf} (1) that 
$\mu_n(s(U\setminus K_n))$ is less than $n^{-1}$. 
Similarly, when $C_n$ contains $U^{-1}$, 
one can see that $\mu_n(r(U\setminus K_n))$ is less than $n^{-1}$. 
Since $\mu_n$ is $K_n$-invariant, 
we have $\mu_n(r(U\cap K_n))=\mu_n(s(U\cap K_n))$. 
Consequently, $\lvert\mu_n(r(U))-\mu_n(s(U))\rvert$ is less than $2/n$ 
for sufficiently large $n$, which implies $\mu(r(U))=\mu(s(U))$. 
Therefore $\mu$ is $G$-invariant. 
\end{proof}

\begin{rem}
Let $G'$ be the isotropy bundle. 
Take $\mu\in M(G)$ arbitrarily. 
In the proof above, 
$U\cap(G'\setminus G^{(0)})$ is contained in $U\setminus K_n$, 
because $G'\setminus G^{(0)}$ does not intersect with $K_n$. 
Hence $\mu(r(U\cap(G'\setminus G^{(0)})))$ is less than $1/n$, 
which implies $\mu(r(U\cap(G'\setminus G^{(0)})))=0$. 
Since $G$ has a countable base of compact open $G$-sets, 
we have $\mu(r(G'\setminus G))=0$ for all $\mu\in M(G)$. 
Assume further that $G$ is minimal. 
From Lemma \ref{Gfull} below, 
$\mu(U)$ is positive for any non-empty open subset $U\subset G^{(0)}$. 
It follows that $r(G'\setminus G^{(0)})$ contains no interior points, and so 
the set of points of $G^{(0)}$ with trivial isotropy is dense in $G^{(0)}$. 
Thus, if $G$ is minimal and almost finite, then 
$G$ is essentially principal. 
\end{rem}

\begin{lem}\label{Blackadar}
Suppose that $G$ is almost finite. 
If two clopen subsets $U,V\subset G^{(0)}$ satisfy 
$\mu(U)<\mu(V)$ for all $\mu\in M(G)$, then 
there exists $\gamma\in[[G]]$ such that $\gamma(U)\subset V$. 
In fact, one can find such $\gamma$ so that 
$\gamma^2=\id$ and $\gamma(x)=x$ for $x\in G^{(0)}\setminus(U\cup\gamma(U))$. 
\end{lem}
\begin{proof}
By removing $U\cap V$ if necessary, 
we may assume that $U$ and $V$ are disjoint. 
Let $C_n$ and $K_n$ be as in Lemma \ref{Mneqempty}. 
Suppose that for each $n\in\N$ 
there exists $\mu_n\in M(K_n)$ such that $\mu_n(U)\geq\mu_n(V)$. 
By taking a subsequence if necessary, 
we may assume that $\mu_n$ converges to $\mu$. 
By the proof of Lemma \ref{Mneqempty}, we have $\mu\in M(G)$. 
This, together with $\mu(U)\geq\mu(V)$, contradicts the assumption. 
It follows that 
there exists $n\in\N$ such that 
$\mu(U)<\mu(V)$ for all $\mu\in M(K_n)$. 
Lemma \ref{cptHopf} (2) applies and yields 
a compact open $K_n$-set $C$ such that $r(C)=U$ and $s(C)\subset V$. 
Set $D=C\cup C^{-1}\cup(G^{(0)}\setminus(r(C)\cup s(C)))$. 
Then $\gamma=\tau_D$ is the desired element. 
\end{proof}

\begin{lem}\label{Gfull}
Suppose that $G$ is almost finite. 
For a clopen subset $U\subset G^{(0)}$, 
the following are equivalent. 
\begin{enumerate}
\item $U$ is $G$-full. 
\item There exists $c>0$ such that $\mu(U)>c$ for all $\mu\in M(G)$. 
\item $\mu(U)>0$ for every $\mu\in M(G)$. 
\end{enumerate}
In particular, if $G$ is minimal, then $\mu(U)>0$ 
for any non-empty clopen subset $U\subset G^{(0)}$ and $\mu\in M(G)$. 
\end{lem}
\begin{proof}
(1)$\Rightarrow$(2). 
Suppose that a clopen subset $U\subset G^{(0)}$ is $G$-full. 
By Lemma \ref{tower1}, there exists $f\in C(Y,\Z)$ such that 
$G$ and $(G|Y)_f$ are isomorphic. 
Put $n=\max\{f(x)\mid x\in U\}$.
Then $1=\mu(G^{(0)})\leq(n+1)\mu(Y)$ for any $\mu\in M(G)$, 
which implies (2). 

(2)$\Rightarrow$(3) is trivial. 

(3)$\Rightarrow$(1). 
We need the hypothesis of almost finiteness for this implication. 
Suppose that $U$ is not $G$-full. 
Let $V=r(s^{-1}(U))$. 
Then $V$ is an open subset such that 
$U\subset V\neq G^{(0)}$ and $r(s^{-1}(V))=V$. 
Let $C_n$ and $K_n$ be as in Lemma \ref{Mneqempty}. 
Take $x\in G^{(0)}\setminus V$ and put 
\[
\mu_n=\frac{1}{\lvert K_n(x)\rvert}\sum_{y\in K_n(x)}\delta_y, 
\]
where $\delta_y$ is the Dirac measure on $y$. 
Then $\mu_n$ is $K_n$-invariant and $\mu_n(V)=0$. 
By taking a subsequence if necessary, 
we may assume that $\mu_n$ converges to $\mu$. 
From the proof of Lemma \ref{Mneqempty}, $\mu$ is $G$-invariant. 
But $\mu(U)\leq\mu(V)=\lim\mu_n(V)=0$, which contradicts (3). 
\end{proof}

\begin{lem}\label{positive}
Suppose that $G$ is almost finite and minimal. 
For $f\in C(G^{(0)},\Z)$, one has 
$[f]\in H_0(G)^+\setminus\{0\}$ if and only if 
$\hat\mu([f])>0$ for every $\mu\in M(G)$. 
\end{lem}
\begin{proof}
The `only if' part easily follows from the lemma above. 
We prove the `if' part. 
Let $C_n$ and $K_n$ be as in Lemma \ref{Mneqempty}. 
In the same way as above, we can find $n\in\N$ such that 
$\displaystyle\int f\,d\mu>0$ for every $\mu\in M(K_n)$. 
By Lemma \ref{cptHopf} (4), 
$[f]$ is in $H_0(K_n)^+$, and hence in $H_0(G)^+$. 
\end{proof}

For the terminologies about ordered abelian groups 
in the following statement, 
we refer the reader to \cite{Ro}. 

\begin{prop}\label{Riesz}
Suppose that $G$ is almost finite and minimal. 
Then $(H_0(G),H_0(G)^+)$ is 
a simple, weakly unperforated, ordered abelian group 
with the Riesz interpolation property. 
\end{prop}
\begin{proof}
By virtue of Lemma \ref{positive}, 
one can see that $(H_0(G),H_0(G)^+)$ is 
a simple, weakly unperforated, ordered abelian group. 
We would like to check the Riesz interpolation property. 
To this end, take $f_1,f_2,g_1,g_2\in C(G^{(0)},\Z)$ 
satisfying $[f_i]\leq[g_j]$ for $i,j=1,2$. 
It suffices to find $h\in C(G^{(0)},\Z)$ 
such that $[f_i]\leq[h]\leq[g_j]$ for $i,j=1,2$. 
Clearly we may assume $[f_i]\neq[g_j]$ for $i,j=1,2$. 
Therefore $g_j-f_i$ is in $H_0(G)^+\setminus\{0\}$ for any $i,j=1,2$. 
By Lemma \ref{positive}, 
we get $\hat\mu([f_i])<\hat\mu([g_j])$ for any $i,j=1,2$. 
Let $C_n$ and $K_n$ be as in Lemma \ref{Mneqempty}. 
In the same way as above, we can find $n\in\N$ such that 
\[
\int f_i\,d\mu<\int g_j\,d\mu
\]
for every $\mu\in M(K_n)$ and $i,j=1,2$. 
By Lemma \ref{cptKakutani}, 
there exists a $K_n$-full clopen subset $Y\subset K_n^{(0)}=G^{(0)}$ 
such that $K_n|Y=Y$. 
Define $h\in C(G^{(0)},\Z)$ by 
\[
h(y)=\max\left\{\sum_{x\in K_n(y)}f_1(x),\sum_{x\in K_n(y)}f_2(x)\right\}
\]
for $y\in Y$ and $h(z)=0$ for $z\notin Y$. 
It is not so hard to see 
\[
\int f_i\,d\mu\leq\int h\,d\mu\leq\int g_j\,d\mu
\]
for every $\mu\in M(K_n)$, 
which implies $[f_i]\leq[h]\leq[g_j]$ in $H_0(K_n)$ 
by Lemma \ref{cptHopf} (4). 
Hence we obtain the same inequalities in $H_0(G)$. 
\end{proof}

\subsection{Equivalence between clopen subsets}

By using the lemmas above, 
we can prove the following two theorems 
concerning the `equivalence' of clopen subsets 
under the action of the (topological) full groups. 

\begin{thm}\label{Hopf1}
Suppose that $G$ is almost finite and minimal. 
For two clopen subsets $U,V\subset G^{(0)}$, 
the following are equivalent. 
\begin{enumerate}
\item There exists $\gamma\in[G]$ such that $\gamma(U)=V$. 
\item $\mu(U)$ equals $\mu(V)$ for all $\mu\in M(G)$. 
\end{enumerate}
Moreover, $\gamma$ in the first condition can be chosen so that 
$\gamma^2=\id$ and $\gamma(x)=x$ for $x\in X\setminus(U\cup V)$. 
\end{thm}
\begin{proof}
Since $G$ is minimal, 
$G^{(0)}$ is either a finite set or a Cantor set. 
If $G^{(0)}$ is a finite set, then the assertion is trivial, 
and so we may assume that $G^{(0)}$ is a Cantor set. 
Any $\mu\in M(G)$ has no atoms, 
because every $G$-orbit is an infinite set. 
By Lemma \ref{Gfull}, 
$\mu(U)$ is positive for any non-empty clopen set $U$ and $\mu\in M(G)$. 
Then the assertion follows from 
a simple generalization of the arguments in \cite[Proposition 2.6]{GW} 
by using Lemma \ref{Blackadar} instead of \cite[Lemma 2.5]{GW}. 
See also \cite[Theorem 3.20]{LaO} and its proof. 
\end{proof}

\begin{thm}\label{Hopf2}
Suppose that $G$ is almost finite. 
For two $G$-full clopen subsets $U,V\subset G^{(0)}$, 
the following are equivalent. 
\begin{enumerate}
\item There exists $\gamma\in[[G]]$ such that $\gamma(U)=V$. 
\item $[1_U]$ equals $[1_V]$ in $H_0(G)$. 
\end{enumerate}
Moreover, $\gamma$ in the first condition can be chosen so that 
$\gamma^2=\id$ and $\gamma(x)=x$ for $x\in X\setminus(U\cup V)$. 
\end{thm}
\begin{proof}
Recall that 
$H_0(G)$ is the quotient of $C(G^{(0)},\Z)$ by $\Ima\delta_1$, 
where $\delta_1:C_c(G,\Z)\to C(G^{(0)},\Z)$ is given by 
\[
\delta_1(f)(x)=s_*(f)(x)-r_*(f)(x)
=\sum_{g\in s^{-1}(x)}f(g)-\sum_{g\in r^{-1}(x)}f(g). 
\]
Suppose that there exists $\gamma\in[[G]]$ such that $\gamma(U)=V$. 
Let $O$ be a compact open $G$-set such that $\gamma=\tau_O$. 
It is easy to see $\delta_1(1_{O\cap r^{-1}(V)})=1_U-1_V$, 
which implies $[1_U]=[1_V]$ in $H_0(G)$. 

We would like to show the other implication (2)$\Rightarrow$(1). 
Clearly we may assume that $U$ and $V$ are disjoint. 
Suppose that there exists $f\in C_c(G,\Z)$ such that $\delta_1(f)=1_U-1_V$. 
For a compact open $G$-set $C$, one has $\delta_1(1_{C^{-1}})=-1_C$. 
Since $G$ has a base of compact open $G$-sets, 
we may assume that 
there exist compact open $G$-sets $C_1,C_2,\dots,C_n$ such that 
$f=1_{C_1}+1_{C_2}+\dots+1_{C_n}$. 
By Lemma \ref{Gfull}, there exists $\ep>0$ such that 
$\mu(U)>\ep$ and $\mu(V)>\ep$ for all $\mu\in M(G)$. 
Almost finiteness of $G$ yields 
an elementary subgroupoid $K\subset G$ such that 
\[
\frac{\lvert C_iKx\setminus Kx\rvert}{\lvert K(x)\rvert}
<\frac{\ep}{n}\quad\text{and}\quad
\frac{\lvert C_i^{-1}Kx\setminus Kx\rvert}{\lvert K(x)\rvert}
<\frac{\ep}{n}
\]
for all $x\in G^{(0)}$ and $i=1,2,\dots,n$. 
Moreover, by the proof of Lemma \ref{Mneqempty}, 
we may further assume 
\[
\frac{\lvert U\cap K(x)\rvert}{\lvert K(x)\rvert}>\ep\quad\text{and}\quad
\frac{\lvert V\cap K(x)\rvert}{\lvert K(x)\rvert}>\ep
\]
for all $x\in G^{(0)}$. 
It follows from $\lvert C_iKx\setminus Kx\rvert
=\lvert s(C_i\setminus K)\cap K(x)\rvert$ that 
\[
\lvert U\cap K(x)\rvert>\sum_{i=1}^n\lvert s(C_i\setminus K)\cap K(x)\rvert
\]
for all $x\in G^{(0)}$. 
Likewise we have 
\[
\lvert V\cap K(x)\rvert>\sum_{i=1}^n\lvert r(C_i\setminus K)\cap K(x)\rvert
\]
for all $x\in G^{(0)}$. 
By Lemma \ref{cptHopf} (2), 
there exist compact open $K$-sets $A_1,A_2,\dots,A_n$ such that 
\[
r(A_i)=s(C_i\setminus K),\quad s(A_i)\subset U\quad\forall i=1,2,\dots,n
\]
and $s(A_i)$'s are mutually disjoint. 
Similarly there exists compact open $K$-sets $B_1,B_2,\dots,B_n$ such that 
\[
s(B_i)=r(C_i\setminus K),\quad r(B_i)\subset V\quad\forall i=1,2,\dots,n
\]
and $r(B_i)$'s are mutually disjoint. 
Then 
\[
D=\bigcup_{i=1}^nB_i(C_i\setminus K)A_i
\]
is a compact open $G$-set such that $r(D)\subset V$ and $s(D)\subset U$. 
Moreover, for any $x\in G^{(0)}$, 
\begin{align*}
\sum_{y\in K(x)}1_{U\setminus s(D)}(y)
&=\sum_{y\in K(x)}1_U(y)-1_{s(D)}(y) \\
&=\sum_{y\in K(x)}\left(1_U(y)-\sum_{i=1}^n1_{s(C_i\setminus K)}(y)\right) \\
&=\sum_{y\in K(x)}
\left(1_U(y)-\sum_{i=1}^ns_*(1_{C_i\setminus K})(y)\right) \\
&=\sum_{y\in K(x)}
\left(1_V(y)-\sum_{i=1}^nr_*(1_{C_i\setminus K})(y)\right) \\
&=\sum_{y\in K(x)}\left(1_V(y)-\sum_{i=1}^n1_{r(C_i\setminus K)}(y)\right)
=\sum_{y\in K(x)}1_{V\setminus r(D)}(y), 
\end{align*}
and so one can find a compact open $K$-set $E$ such that 
$s(E)=U\setminus s(D)$ and $r(E)=V\setminus r(D)$ by Lemma \ref{cptHopf} (3). 
Hence $F=D\cup E$ is a compact open $G$-set 
satisfying $s(F)=U$ and $r(F)=V$. 
Define a compact open $G$-set $O$ 
by $O=F\cup F^{-1}\cup(G^{(0)}\setminus(U\cup V))$. 
Then $\gamma=\tau_O\in[[G]]$ is a desired element. 
\end{proof}

\begin{rem}
In the light of Proposition \ref{exact}, 
the two conditions of the theorem above are also equivalent to 
\begin{enumerate}
\item[(3)] There exists $w\in N(C(G^{(0)}),C^*_r(G))$ such that 
$w1_Uw^*=1_V$. 
\end{enumerate}
We do not know when this is equivalent to the condition that 
the two projections $1_U$ and $1_V$ have the same class in $K_0(C^*_r(G))$. 
\end{rem}

We also remark that 
a special case of Theorem \ref{Hopf2} is implicitly contained 
in the proof of \cite[Theorem 3.16]{LaO}.

\section{The index map}

In this section, 
we introduce a group homomorphism, called the index map, 
from $[[G]]$ to $H_1(G)$. 
When $G$ is almost finite, it will be shown that 
the index map is surjective (Theorem \ref{surjective}) and 
that any element in the kernel of the index map can be written 
as a product of four elements of finite order (Theorem \ref{injective}). 

Throughout this section, 
we let $G$ be a second countable \'etale essentially principal groupoid 
whose unit space is compact and totally disconnected. 
For $f\in C_c(G,\Z)$, 
we denote its equivalence class in $H_1(G)$ by $[f]$. 

\begin{df}
For $\gamma\in[[G]]$, 
a compact open $G$-set $U$ satisfying $\gamma=\tau_U$ uniquely exists, 
because $G$ is essentially principal. 
It is easy to see that $1_U$ is in $\ker\delta_1$. 
We define a map $I:[[G]]\to H_1(G)$ by $I(\gamma)=[1_U]$ 
and call it the index map. 
\end{df}

\begin{rem}
When $G$ arises from a minimal homeomorphism on a Cantor set, 
$H_1(G)$ is $\Z$ and 
the above definition agrees with that in \cite[Section 5]{GPS2}. 
In this case, the index map can be understood 
through the Fredholm index of certain Fredholm operators. 
\end{rem}

\begin{lem}\label{H1}
\begin{enumerate}
\item If $U,U'\subset G$ are compact open $G$-sets satisfying $s(U)=r(U')$, 
then 
\[
O=\{(g,g')\in G^{(2)}\mid g\in U,\ g'\in U'\}
\]
is a compact open subset of $G^{(2)}$ and 
$\delta_2(1_O)=1_U-1_{UU'}+1_{U'}$. 
\item The index map $I:[[G]]\to H_1(G)$ is a homomorphism. 
\item $[1_U]=0$ for any clopen subset $U\subset G^{(0)}$. 
\item $[1_U+1_{U^{-1}}]=0$ for any compact open $G$-set $U\subset G$. 
\end{enumerate}
\end{lem}
\begin{proof}
(1) follows from a straightforward computation. 
(2), (3) and (4) are direct consequences of (1). 
\end{proof}

\subsection{Surjectivity of the index map}

In this subsection we prove that 
the index map $I:[[G]]\to H_1(G)$ is surjective when $G$ is almost finite 
(Theorem \ref{surjective}). 
To this end, we need the following lemma. 

\begin{lem}
Let $K$ be an elementary subgroupoid of $G$ and 
let $Y\subset G^{(0)}$ be a $K$-full clopen subset such that $K|Y=Y$. 
Suppose that $f\in C_c(G,\Z)$ is in $\ker\delta_1$. 
Then, the function $\tilde f\in C_c(G,\Z)$ defined by 
\[
\tilde f(g)
=\begin{cases}\sum_{g_1,g_2\in K}f(g_1gg_2) & g\in G|Y \\
0 & \text{otherwise}\end{cases}
\]
is also in $\ker\delta_1$ and $[f]=[\tilde f]$ in $H_1(G)$. 
\end{lem}
\begin{proof}
Put $f_0=s_*(f)=r_*(f)\in C(G^{(0)},\Z)$. 
Define $k,\bar{k}\in C(K,\Z)$ by 
\[
k(g)=\begin{cases}f_0(s(g)) & r(g)\in Y \\
0 & r(g)\notin Y\end{cases}
\]
and $\bar{k}(g)=k(g^{-1})$ for $g\in K$. 
By Lemma \ref{H1} (4), $[k+\bar{k}]=0$ in $H_1(K)$ 
and hence in $H_1(G)$. 
We define $h_1\in C_c(G^{(2)},\Z)$ by 
\[
h_1(g,g')=\begin{cases}f(g') & g\in K,\ r(g)\in Y\\
0 & \text{otherwise.}\end{cases}
\]
Letting $d_i:G^{(2)}\to G$ ($i=0,1,2$) be the maps introduced in Section 3, 
we have 
\begin{align*}
\delta_2(h_1)(g)
&=d_{0*}(h_1)(g)-d_{1*}(h_1)(g)+d_{2*}(h_1)(g) \\
&=\sum_{g_0\in G}h_1(g_0,g)-d_{1*}(h_1)(g)+\sum_{g_0\in G}h_1(g,g_0) \\
&=f(g)-d_{1*}(h_1)(g)+k(g). 
\end{align*}
Define $h_2\in C_c(G^{(2)},\Z)$ by 
\[
h_2(g,g')=\begin{cases}d_{1*}(h_1)(g) & r(g)\in Y,\ g'\in K,\ s(g')\in Y \\
0 & \text{otherwise.}\end{cases}
\]
Then 
\begin{align*}
d_{0*}(h_2)(g)
&=\sum_{g_0\in G}h_2(g_0,g) \\
&=\begin{cases}\sum_{r(g_0)\in Y}d_{1*}(h_1)(g_0) & g\in K,\ s(g)\in Y \\
0 & \text{otherwise}\end{cases} \\
&=\begin{cases}s_*(f)(r(g)) & g\in K,\ s(g)\in Y \\
0 & \text{otherwise}\end{cases} \\
&=\begin{cases}f_0(r(g)) & g\in K,\ s(g)\in Y \\
0 & \text{otherwise}\end{cases} \\
&=\bar{k}(g). 
\end{align*}
Moreover, it is easy to see 
$d_{1*}(h_2)=\tilde f$ and $d_{2*}(h_2)=d_{1*}(h_1)$. 
Hence 
\begin{align*}
\delta_2(h_1)+\delta_2(h_2)
&=(f-d_{1*}(h_1)+k)+(\bar{k}-\tilde f+d_{1*}(h_1)) \\
&=f+k+\bar{k}-\tilde f, 
\end{align*}
and so $\tilde f$ is in $\ker\delta_1$ and $[f]=[\tilde f]$ in $H_1(G)$. 
\end{proof}

\begin{thm}\label{surjective}
When $G$ is almost finite, the index map $I$ is surjective. 
\end{thm}
\begin{proof}
Take $f\in C_c(G,\Z)$ such that $\delta_1(f)=0$. 
We will show that 
there exists $\gamma\in[[G]]$ satisfying $I(\gamma)=[f]$. 
By Lemma \ref{H1} (4), we may assume $f(g)\geq0$ for all $g\in G$. 
Since $G$ has a base of compact open $G$-sets, 
there exist compact open $G$-sets $C_1,C_2,\dots,C_n$ such that 
$f=1_{C_1}+1_{C_2}+\dots+1_{C_n}$. 
Almost finiteness of $G$ yields 
an elementary subgroupoid $K\subset G$ such that 
\[
\frac{\lvert C_iKx\setminus Kx\rvert}{\lvert K(x)\rvert}
<\frac{1}{n}\quad\text{and}\quad
\frac{\lvert C_i^{-1}Kx\setminus Kx\rvert}{\lvert K(x)\rvert}
<\frac{1}{n}
\]
for all $x\in G^{(0)}$ and $i=1,2,\dots,n$. 
For $x\in G^{(0)}$, let $E_x=\{g\in G\mid g\notin K,\ s(g)\in K(x)\}$. 
Then
\[
\sum_{g\in E_x}f(g)=\sum_{i=1}^n\sum_{g\in E_x}1_{C_i}(g)
=\sum_{i=1}^n\lvert C_iKx\setminus Kx\rvert
<\sum_{i=1}^nn^{-1}\lvert K(x)\rvert=\lvert K(x)\rvert
\]
for any $x\in G^{(0)}$. 
Likewise, we have $\sum_{g\in E_x}f(g^{-1})<\lvert K(x)\rvert$. 

By Lemma \ref{cptKakutani}, 
there exists a $K$-full clopen subset $Y\subset K^{(0)}=G^{(0)}$ 
such that $K|Y=Y$. 
Let $\tilde f$ be as in the preceding lemma. 
Define $f_0\in C(G^{(0)},\Z)$ by $f_0=\tilde f|G^{(0)}$. 
By Lemma \ref{H1} (3), $[\tilde f]=[\tilde f-f_0]$ in $H_1(G)$. 
Since $G$ has a base of compact open $G$-sets, 
we may assume that there exist compact open $G$-sets 
$D_1,D_2,\dots,D_m\subset G\setminus G^{(0)}$ such that 
\[
\tilde f-f_0=1_{D_1}+1_{D_2}+\dots+1_{D_m}. 
\]
Notice that 
$r(D_i)$ and $s(D_i)$ are contained in $Y$ and 
that $D_i$ does not intersect with $K$. 
For any $y\in Y$ one has 
\begin{align*}
\sum_{i=1}^m\lvert s(D_i)\cap K(y)\rvert
&=\sum_{i=1}^m1_{s(D_i)}(y)
=\sum_{i=1}^ms_*(1_{D_i})(y)
=s_*(\tilde f-f_0)(y) \\
&=\sum_{s(g)=y,\ g\notin K}\tilde f(g)
=\sum_{g\in E_y}f(g)<\lvert K(y)\rvert. 
\end{align*}
It follows from Lemma \ref{cptHopf} (2) that 
there exist compact open $K$-sets $A_1,A_2,\dots,A_m$ such that 
$r(A_i)=s(D_i)$ for all $i=1,2,\dots,m$ and 
$s(A_i)$'s are mutually disjoint. 
In a similar way, we also have 
\[
\sum_{i=1}^m\lvert r(D_i)\cap K(y)\rvert<\lvert K(y)\rvert, 
\]
and so there exist compact open $K$-sets $B_1,B_2,\dots,B_m$ such that 
$s(B_i)=r(D_i)$ for all $i=1,2,\dots,m$ and 
$r(B_i)$'s are mutually disjoint. 
Besides, from $s_*(\tilde f-f_0)=r_*(\tilde f-f_0)$, we get 
\[
\sum_{i=1}^m1_{r(A_i)}=\sum_{i=1}^m1_{s(D_i)}
=\sum_{i=1}^m1_{r(D_i)}=\sum_{i=1}^m1_{s(B_i)}, 
\]
which implies 
\begin{align*}
\left\lvert\bigcup_{i=1}^ms(A_i)\cap K(x)\right\lvert
&=\sum_{i=1}^m\lvert s(A_i)\cap K(x)\rvert
=\sum_{i=1}^m\lvert r(A_i)\cap K(x)\rvert \\
&=\sum_{i=1}^m\lvert s(B_i)\cap K(x)\rvert
=\sum_{i=1}^m\lvert r(B_i)\cap K(x)\rvert
=\left\lvert\bigcup_{i=1}^mr(B_i)\cap K(x)\right\rvert
\end{align*}
for $x\in G^{(0)}$, because $A_i$ and $B_i$ are $K$-sets. 
Hence, by Lemma \ref{cptHopf} (3), we may replace $A_i$ and assume 
\[
\bigcup_{i=1}^ms(A_i)=\bigcup_{i=1}^mr(B_i). 
\]
Then 
\[
E=\bigcup_{i=1}^mB_iD_iA_i
\]
is a compact open $G$-set satisfying $s(E)=r(E)$. 
Let $k=\sum_{i=1}^m1_{A_i}+1_{B_i}\in C(K,\Z)$. 
It is easy to verify that $k$ is in $\ker\delta_1$. 
Therefore, by Lemma \ref{cptHomology}, 
$[k]$ is zero in $H_1(K)$ and hence in $H_1(G)$. 
By Lemma \ref{H1} (1), 
$1_{B_i}-1_{B_iD_i}+1_{D_i}$ and $1_{B_iD_i}-1_{B_iD_iA_i}+1_{A_i}$ 
are zero in $H_1(G)$ for all $i=1,2,\dots,m$. 
Consequently, 
\[
[f]=[\tilde f]=[\tilde f-f_0]=[1_{D_1}+\dots+1_{D_m}]
=[1_{D_1}+\dots+1_{D_m}]+[k]=[1_E]
\]
in $H_1(G)$. 
Let $F=E\cup(G^{(0)}\setminus s(E))$. 
Then $F$ is a compact open $G$-set 
satisfying $r(F)=s(F)=G^{(0)}$ and $[1_F]=[1_E]$. 
Thus $\gamma=\tau_F$ is in $[[G]]$ and $I(\gamma)=[f]$. 
\end{proof}

\subsection{Kernel of the index map}

Next, we would like to determine the kernel of the index map. 

\begin{df}
\begin{enumerate}
\item We say that $\gamma\in\Homeo(G^{(0)})$ is elementary, 
if $\gamma$ is of finite order and 
$\{x\in G^{(0)}\mid\gamma^k(x)=x\}$ is clopen for any $k\in\N$. 
\item We let $[[G]]_0$ denote the subgroup of $[[G]]$ 
which is generated by all elementary homeomorphisms in $[[G]]$. 
Evidently $[[G]]_0$ is a normal subgroup of $[[G]]$. 
\end{enumerate}
\end{df}

\begin{lem}\label{elementary}
\begin{enumerate}
\item When $G$ is principal, 
$\gamma\in[[G]]$ is elementary if and only if $\gamma$ is of finite order. 
\item $\gamma\in[[G]]$ is elementary if and only if 
there exists an elementary subgroupoid $K\subset G$ 
such that $\gamma\in[[K]]$. 
\item If $\gamma\in[[G]]$ is elementary, then $I(\gamma)=0$. 
In particular, $\ker I$ contains $[[G]]_0$. 
\end{enumerate}
\end{lem}
\begin{proof}
(1) This is clear from the definition. 

(2) The `if' part follows from 
\cite[Proposition 3.2]{Msrtfg} and its proof. 
Let us show the `only if' part. 
Suppose that $\gamma=\tau_U\in[[G]]$ is elementary. 
There exists $n\in\N$ such that $\gamma^n=\id$. 
Then $K=(U\cup G^{(0)})^n$ is a compact open subgroupoid of $G$. 
Since the fixed points of $\gamma^k$ form a clopen set for any $k\in\N$, 
$K$ is principal. 
It follows from $U\subset K$ that $\gamma$ belongs to $[[K]]$. 

(3) This readily follows from (2) and Lemma \ref{cptHomology}. 
\end{proof}

\begin{rem}
Even if $\gamma\in[[G]]$ is of finite order, 
$I(\gamma)$ is not necessarily zero. 
Let $\phi:\Z/N\Z\curvearrowright X$ be an action of $\Z/N\Z$ 
on a Cantor set $X$ by homeomorphisms and 
let $G_\phi$ be the transformation groupoid arising from $\phi$. 
The generator $\gamma$ of $\phi$ is clearly in $[[G_\phi]]$ and 
of finite order. 
It is well-known that 
$H_1(G_\phi)\cong H_1(\Z/N\Z,C(X,\Z))$ is isomorphic to 
\[
\{f\in C(X,\Z)\mid f=f\circ\gamma\}/
\{f+f\circ\gamma+\dots+f\circ\gamma^{N-1}\mid f\in C(X,\Z)\}. 
\]
Hence, when $\phi$ is not free, $I(\gamma)$ is not zero in $H_1(G_\phi)$. 
\end{rem}

\begin{rem}
In Corollary \ref{bijective}, it will be shown that 
$[[G]]/[[G]]_0$ is isomorphic to $H_1(G)$ via the index map, 
when $G$ is almost finite and principal. 
This, however, does not mean that $H_1(G)$ is always torsion free. 
Indeed, it was shown in \cite[Section 6.4]{GHK} that 
the dual canonical $D_6$ tiling contains $2$-torsions in its $H_1$-group, 
and so there exists a free action $\phi$ of $\Z^3$ on a Cantor set 
by homeomorphisms such that $H_1(G_\phi)$ contains $2$-torsions. 
Note that $G_\phi$ is almost finite by Lemma \ref{ZNisAF}. 
\end{rem}

In order to prove Theorem \ref{injective}, we need a series of lemmas. 

\begin{lem}\label{decomp1}
Suppose that $G$ is almost finite. 
For any $\gamma\in[[G]]$, 
there exist an elementary homeomorphism $\gamma_0\in[[G]]$ and 
a clopen subset $V\subset G^{(0)}$ such that 
$\gamma_0\gamma(x)=x$ for any $x\in V$ and 
$\mu(V)\geq1/2$ for any $\mu\in M(G)$. 
\end{lem}
\begin{proof}
Take a compact open $G$-set $U$ satisfying $\gamma=\tau_U$. 
Since $G$ is almost finite, 
there exists an elementary subgroupoid $K\subset G$ such that 
\[
\lvert UKx\setminus Kx\rvert<2^{-1}\lvert K(x)\rvert
\]
for all $x\in G^{(0)}$. 
Let $V=s(U\cap K)$. 
Then 
\[
\lvert K(x)\cap V\rvert=\lvert K(x)\rvert-\lvert UKx\setminus Kx\rvert
\geq2^{-1}\lvert K(x)\rvert. 
\]
By Lemma \ref{cptHopf} (1), 
we have $\mu(V)\geq2^{-1}$ for all $\mu\in M(K)$ 
and hence for all $\mu\in M(G)$. 
Moreover, one also has 
\[
\lvert K(x)\cap s(U\setminus K)\rvert
=\lvert K(x)\setminus s(U\cap K)\rvert
=\lvert K(x)\setminus r(U\cap K)\rvert
=\lvert K(x)\cap r(U\setminus K)\rvert
\]
for all $x\in G^{(0)}$. 
It follows from Lemma \ref{cptHopf} (3) that 
there exists a compact open $K$-set $W$ such that 
$s(W)=r(U\setminus K)$ and $r(W)=s(U\setminus K)$. 
Then $O=W\cup(U^{-1}\cap K)$ is a compact open $K$-set 
satisfying $s(O)=r(O)=G^{(0)}$, 
and so $\gamma_0=\tau_O$ is elementary by Lemma \ref{elementary} (2). 
Clearly $\gamma_0\gamma(x)=x$ for $x\in V$, 
which completes the proof. 
\end{proof}

\begin{lem}\label{decomp2}
Suppose that $G$ is almost finite. 
Let $V\subset G^{(0)}$ be a clopen subset and let $\gamma\in\ker I$. 
Suppose that 
$\gamma(x)=x$ for any $x\in V$ and 
$\mu(V)\geq1/2$ for any $\mu\in M(G)$. 
Then, there exist an elementary subgroupoid $K\subset G$ and $\tau_U\in[[G]]$ 
such that $\tau_U\gamma^{-1}$ is elementary and 
\[
\sum_{g_1,g_2\in K}1_U(g_1gg_2)
=\sum_{g_1,g_2\in K}1_U(g_1g^{-1}g_2)
\]
holds for all $g\in G$. 
\end{lem}
\begin{proof}
Let $\gamma=\tau_O$. 
Since $I(\gamma)=0$, $1_O$ is in $\Ima\delta_2$. 
It follows from Lemma \ref{H1} (1) that 
there exist compact open $G$-sets 
$A_1,\dots,A_n$, $B_1,\dots,B_n$, $C_1,\dots,C_m$ and $D_1,\dots,D_m$ 
such that $s(A_i)=r(B_i)$, $s(C_j)=r(D_j)$ and 
\[
1_O=\left(\sum_{i=1}^n1_{A_i}-1_{A_iB_i}+1_{B_i}\right)
-\left(\sum_{j=1}^m1_{C_j}-1_{C_jD_j}+1_{D_j}\right). 
\]
Set 
\[
E=\bigcup_{i=1}^n(A_i\cup B_i\cup A_iB_i)
\cup\bigcup_{j=1}^m(C_j\cup D_j\cup C_jD_j). 
\]
Almost finiteness of $G$ yields 
an elementary subgroupoid $K\subset G$ such that 
\[
\lvert (E\cup E^{-1})Kx\setminus Kx\rvert
<\frac{1}{18(n{+}m)}\lvert K(x)\rvert
\]
for all $x\in G^{(0)}$. 
Since $\mu(V)$ is not less than $1/2$ for every $\mu\in M(G)$, 
by the proof of Lemma \ref{Mneqempty}, 
we may further assume 
\[
\lvert V\cap K(x)\rvert>3^{-1}\lvert K(x)\rvert
\]
for all $x\in G^{(0)}$. 
Define compact open $G$-sets $A_i'$, $B_i'$, $C_j'$, $D_j'$ by 
\[
A_i'=A_i\setminus((A_i\cap K)r(B_i\cap K)),\quad 
B_i'=B_i\setminus(s(A_i\cap K)(B_i\cap K)), 
\]
\[
C_j'=C_j\setminus((C_j\cap K)r(D_j\cap K)),\quad 
D_j'=D_j\setminus(s(C_j\cap K)(D_j\cap K)). 
\]
Then $s(A_i')=r(B_i')$, $s(C_j')=r(D_j')$ and 
\[
1_O=\left(\sum_{i=1}^n1_{A'_i}-1_{A'_iB'_i}+1_{B'_i}\right)
-\left(\sum_{j=1}^m1_{C'_j}-1_{C'_jD'_j}+1_{D'_j}\right)+k
\]
for some $k\in C(K,\Z)$. 
In addition, 
since 
\[
(A_i\cap K)r(B_i\cap K)=(A_i\cap K)\cap(A_iB_i\cap K)B_i^{-1}, 
\]
we have 
\begin{align*}
\lvert r(A'_i)\cap K(x)\rvert
&\leq\lvert r(A_i\setminus(A_i\cap K))\cap K(x)\rvert
+\lvert r(A_i\setminus(A_iB_i\cap K)B_i^{-1})\cap K(x)\rvert \\
&=\lvert A_i^{-1}Kx\setminus Kx\rvert
+\lvert (A_iB_i)^{-1}Kx\setminus Kx\rvert \\
&<\frac{1}{9(n{+}m)}\lvert K(x)\rvert
\end{align*}
for any $x\in G^{(0)}$. 
Similar estimates can be obtained 
for $s(A'_i)$, $s(B'_i)$, $r(C'_j)$, $s(C'_j)$ and $s(D'_j)$. 
By Lemma \ref{cptHopf} (2), 
we can find compact open $K$-sets $P_{k,i}$ ($k=1,2,3$, $i=1,2,\dots,n$) and 
$Q_{l,j}$ ($l=1,2,3$, $j=1,2,\dots,m$) such that 
\[
s(P_{1,i})=r(A'_i),\quad s(P_{2,i})=s(A'_i),\quad s(P_{3,i})=s(B'_i)
\]
\[
s(Q_{1,j})=r(C'_j),\quad s(Q_{2,j})=s(C'_j),\quad s(Q_{3,j})=s(D'_j)
\]
and the ranges of $P_{k,i}$'s and $Q_{l,j}$'s are 
mutually disjoint and contained in $V$. 
Define compact open $G$-sets $A''_i$, $B''_i$, $C''_j$, $D''_j$ by 
\[
A''_i=P_{1,i}A'_iP_{2,i}^{-1},\quad 
B''_i=P_{2,i}B'_iP_{3,i}^{-1},\quad 
C''_j=Q_{1,j}C'_jQ_{2,j}^{-1},\quad 
D''_i=Q_{2,j}D'_iQ_{3,j}^{-1}. 
\]
Then 
\[
F=\bigcup_{i=1}^n(A''_i\cup B''_i\cup(A''_iB''_i)^{-1})^{-1}
\cup\bigcup_{j=1}^m(C''_j\cup D''_j\cup(C''_jD''_j)^{-1})
\]
is a compact open $G$-set satisfying $s(F)=r(F)=F^3\subset V$. 
Moreover, $\tau_F$ and $\tau_F^2$ have no fixed points. 
Set $F_0=G^{(0)}\setminus s(F)$ and $\widetilde F=F\cup F_0$. 
Then $\tau_{\widetilde F}$ is an elementary homeomorphism in $[[G]]$. 
Define a compact open subset $U\subset G$ by $U=F\cup OF_0$. 
Since $s(F)$ is contained in $V$, 
$U$ is a $G$-set and $\tau_{\tilde F}\tau_O=\tau_U$. 
Furthermore, 
\begin{align*}
1_U&=1_F+1_{OF_0}=1_F+1_O-1_{s(F)} \\
&=1_F+\left(\sum_{i=1}^n1_{A'_i}-1_{A'_iB'_i}+1_{B'_i}\right)
-\left(\sum_{j=1}^m1_{C'_j}-1_{C'_jD'_j}+1_{D'_j}\right)+k-1_{s(F)}. 
\end{align*}
It is not so hard to check 
\[
\sum_{g_1,g_2\in K}1_U(g_1gg_2)
=\sum_{g_1,g_2\in K}1_U(g_1g^{-1}g_2). 
\]
\end{proof}

\begin{lem}\label{decomp3}
Let $K$ be an elementary subgroupoid of $G$ and 
let $\gamma{=}\tau_U\in[[G]]$. 
If 
\[
\sum_{g_1,g_2\in K}1_U(g_1gg_2)
=\sum_{g_1,g_2\in K}1_U(g_1g^{-1}g_2)
\]
holds for all $g\in G$, 
then there exists $\gamma_0\in[[G]]$ such that 
$\gamma_0^2\in[[K]]$ and $\gamma_0\gamma\in[[K]]$. 
\end{lem}
\begin{proof}
By Lemma \ref{cptKakutani}, 
there exists a $K$-full clopen subset $Y\subset K^{(0)}=G^{(0)}$ 
such that $K|Y=Y$. 
If $C$ is a compact open $G|Y$-set, then 
\[
KCK=\{g_1gg_2\in G\mid g_1g_2,\in K,\ g\in C\}
\]
is a compact open subset of $G$. 
Since $G|Y$ is written as a disjoint union of compact open $G|Y$-sets, 
there exist mutually disjoint compact open $G|Y$-sets $C_1,C_2,\dots,C_n$ 
such that $U\cup U^{-1}$ is contained in $\bigcup_iKC_iK$. 
Note that $KC_iK$'s are also mutually disjoint, because of $K|Y=Y$. 
Define a (possibly empty) compact open $G$-set $D_{i,j}$ by 
\[
D_{i,j}=U\cap KC_iK\cap KC_j^{-1}K=U\cap K(C_i\cap C_j^{-1})K, 
\]
so that $1_U=\sum_{i,j}1_{D_{i,j}}$. 
Take $i,j\in\{1,2,\dots,n\}$ and $y\in Y$ arbitrarily. 
If $r(C_i\cap C_j^{-1})$ does not contain $y$, then clearly 
\[
r(D_{i,j})\cap K(y)=\emptyset=s(D_{j,i})\cap K(y). 
\]
Suppose that $r(C_i\cap C_j^{-1})$ contains $y$. 
There exists a unique element $g\in C_i\cap C_j^{-1}$ such that $r(g)=y$ and 
one has 
\begin{align*}
\lvert r(D_{i,j})\cap K(y)\rvert
&=\lvert\{g_1gg_2\in U\mid g_1,g_2\in K\}\rvert \\
&=\sum_{g_1,g_2\in K}1_U(g_1gg_2) \\
&=\sum_{g_1,g_2\in K}1_U(g_1g^{-1}g_2) \\
&=\lvert\{g_1g^{-1}g_2\in U\mid g_1,g_2\in K\}\rvert
=\lvert s(D_{j,i})\cap K(y)\rvert. 
\end{align*}
It follows from Lemma \ref{cptHopf} (3) that 
there exists a compact open $K$-set $A_{i,j}$ 
such that $s(A_{i,j})=r(D_{i,j})$ and $r(A_{i,j})=s(D_{j,i})$. 
Set 
\[
B=\bigcup_{i,j=1}^nD_{j,i}A_{i,j}. 
\]
It is easy to check that 
$B$ is a compact open $G$-set satisfying $r(B)=s(B)=G^{(0)}$. 
Furthermore, one can see that 
$D_{j,i}A_{i,j}D_{i,j}$ is a compact open $K$-set for any $i,j$. 
Let $\gamma_0=\tau_B$. 
Then we obtain $\gamma_0^2\in[[K]]$ and $\gamma_0\gamma\in[[K]]$. 
\end{proof}

From Lemma \ref{decomp1}, \ref{decomp2}, \ref{decomp3} and 
Lemma \ref{elementary}, 
we deduce the following theorem. 

\begin{thm}\label{injective}
Suppose that $G$ is almost finite. 
Suppose that $\gamma\in[[G]]$ is in the kernel of the index map. 
Then there exist $\gamma_1,\gamma_2,\gamma_3,\gamma_4\in[[G]]$ such that 
$\gamma=\gamma_1\gamma_2\gamma_3\gamma_4$ and 
$\gamma_1,\gamma_2^2,\gamma_3,\gamma_4$ are elementary. 
In particular, 
$\gamma$ is written as a product of four elements in $[[G]]$ of finite order. 
\end{thm}

\begin{rem}
The theorem above is a generalization of \cite[Lemma 4.1]{Msrtfg}, 
in which it was shown that 
any $\gamma\in[[G_\phi]]\cap\ker I$ can be written as a product of 
two elementary homeomorphisms 
when $\phi$ is a minimal free action of $\Z$ on a Cantor set. 
\end{rem}

\subsection{Conclusions}

We conclude this paper with the following immediate consequences of 
Theorem \ref{surjective} and Theorem \ref{injective}. 

\begin{cor}\label{Phi1}
Suppose that $G$ is almost finite. 
Then there exists a homomorphism $\Phi_1:H_1(G)\to K_1(C^*_r(G))$ 
such that $\Phi_1(I(\gamma))$ is equal to the $K_1$-class of $\rho(\gamma)$ 
for $\gamma\in[[G]]$, 
where $\rho:[[G]]\to N(C(G^{(0)}),C^*_r(G))$ is the homomorphism 
described in Proposition \ref{exact} (3). 
\end{cor}
\begin{proof}
When $I(\gamma)$ is zero, by Theorem \ref{injective}, 
$\gamma$ is a product of four homeomorphisms of finite order. 
If a unitary in a $C^*$-algebra is of finite order, then 
its $K_1$-class is zero. 
Therefore the $K_1$-class of $\rho(\gamma)$ is zero for $\gamma\in\ker I$. 
Since the index map $I:[[G]]\to H_1(G)$ is surjective 
by Theorem \ref{surjective}, 
we can define a homomorphism $\Phi_1:H_1(G)\to K_1(C^*_r(G))$ 
by letting $\Phi_1(I(\gamma))$ be the $K_1$-class of $\rho(\gamma)$. 
\end{proof}

The corollary above says that 
$H_1(G)$ corresponds to a subgroup of $K_1(C^*_r(G))$ 
generated by unitary normalizers of $C(G^{(0)})$. 
We do not know whether the homomorphism $\Phi_1$ is injective or not and 
whether the range of $\Phi_1$ is a direct summand of $K_1(C^*_r(G))$ or not. 

When $G$ is principal, 
combining Theorem \ref{surjective} and Theorem \ref{injective}, 
we obtain the following corollary. 

\begin{cor}\label{bijective}
Suppose that $G$ is almost finite and principal. 
Then the kernel of the index map is equal to $[[G]]_0$, 
and the quotient group $[[G]]/[[G]]_0$ is isomorphic to $H_1(G)$ 
via the index map. 
\end{cor}

\end{document}